\documentclass[11pt]{amsart}
\usepackage{CJK}
\usepackage{indentfirst}
\usepackage{amscd}
\usepackage{amsmath, amssymb}
\usepackage{amsfonts}
\usepackage{enumerate}
\usepackage{color}

\newcommand{\vp}{\varphi}
\newcommand{\ve}{\varepsilon}
\newcommand{\ddbar}{\sqrt{-1} \partial \overline{\partial}}
\newcommand{\ol}{\overline}

\newcommand{\ri}{\rightarrow}

\newcommand{\ov}[1]{\overline{#1}}
\newcommand{\de}{\partial}
\newcommand{\dbar}{\overline{\partial}}
\newcommand{\ti}{\tilde}

\newtheorem*{claim*}{Claim}

\begin{document}
\newcounter{theor}
\setcounter{theor}{1}
\newtheorem{claim}{Claim}
\newtheorem{theorem}{Theorem}[section]
\newtheorem{lemma}[theorem]{Lemma}
\newtheorem{corollary}[theorem]{Corollary}
\newtheorem{proposition}[theorem]{Proposition}
\newtheorem{prop}{Proposition}[section]
\newtheorem{question}{question}[section]
\newtheorem{defn}{Definition}[section]
\newtheorem{remark}{Remark}[section]

\numberwithin{equation}{section}

\title[Fully non-linear elliptic equations]{Fully non-linear elliptic equations on compact almost Hermitian manifolds}

\author[J. Chu]{Jianchun Chu}
\address{School of Mathematical Sciences, Peking University, Yiheyuan Road 5, Beijing, 100871, P. R. China, }
\email{jianchunchu@math.pku.edu.cn}

\author[L. Huang]{Liding Huang}
\address{Westlake Institute for Advanced Study (Westlake University), 18 Shilongshan Road, Cloud Town, Xihu District, Hangzhou 310024, P. R. China}
\email{huangliding@westlake.edu.cn}

\author[J. Zhang]{Jiaogen Zhang}
\address{School of Mathematical Sciences, University of Science and Technology of China, Hefei 230026, P. R. China}
\email{zjgmath@mail.ustc.edu.cn}

\begin{abstract}
In this paper, we establish a priori estimates for solutions of a general class of fully non-linear equations on compact almost Hermitian manifolds. As an application, we solve the complex Hessian equation and the Monge--Amp\`ere equation for $(n-1)$-plurisubharmonic equations in the almost Hermitian setting.
\end{abstract}
\maketitle

\section{Introduction}\label{introduction}

Let $(M,\chi,J)$ be a compact almost Hermitian manifold of real dimension $2n$. Suppose that $\omega$ is a real $(1,1)$-form on $(M,J)$. For $u\in C^{2}(M)$, we write
\[
\omega_{u} := \omega+\ddbar{u}
:= \omega+\frac{1}{2}(dJdu)^{(1,1)}
\]
and let $\mu(u)=(\mu_{1}(u),\ldots,\mu_{n}(u))$ be the eigenvalues of $\omega_{u}$ with respect to $\chi$. For notational convenience, we sometimes denote $\mu_{i}(u)$ by $\mu_{i}$ when no confusion will arise. We consider the following PDE:
\begin{equation}\label{nonlinear equation}
F(\omega_{u}) := f(\mu_{1},\cdots,\mu_{n}) = h,
\end{equation}
where $h\in C^{\infty}(M)$ and $f$ is a smooth symmetric function.

The real version of the equation \eqref{nonlinear equation} has been studied extensively. In the pioneering work \cite{CNS85}, Caffarelli-Nirenberg-Spruck considered the Dirichlet problem in the domain of $\mathbb{R}^n$. In \cite{Guan14}, Guan studied such equations on Riemannian manifolds, introduced a kind of subsolution and derived $C^2$ estimate if subsolution exists. Later, Sz\'ekelyhidi \cite{Szekelyhidi18} focused on the complex setting, proposed another kind of subsolution (i.e. $\mathcal{C}$-subsolution) and established $C^{2,\alpha}$ estimate when $\mathcal{C}$-subsolution exists. In \cite{CM21}, Chu-McCleerey derived the real Hessian estimate independent of $\inf_{M}h$, which can be applied to the degenerate case of \eqref{nonlinear equation}.

Following the setting of Sz\'ekelyhidi \cite{Szekelyhidi18}, we assume that $f$ is defined in an open symmetric cone $\Gamma\subsetneq\mathbb{R}^{n}$ with vertex at the origin containing the positive orthant $\Gamma_{n}=\{\mu\in\mathbb{R}^{n}:\mu_{i}>0 \ \text{for $i=1,\ldots,n$}\}$. Furthermore, suppose that
\begin{enumerate}[(i)]\setlength{\itemsep}{1mm}
\item $f_{i}=\frac{\de f}{\de\mu_{i}}>0$ for all $i$ and $f$ is concave,
\item $\sup_{\partial\Gamma}f<\inf_{M}h$,
\item For any $\sigma<\sup_{\Gamma} f$ and $\mu\in \Gamma$, we have $\lim_{t\rightarrow \infty}f(t\mu)>\sigma$,
\end{enumerate}
where
\[
\sup_{\de\Gamma}f = \sup_{\lambda'\in\de\Gamma}\limsup_{\lambda\rightarrow\lambda'}f(\lambda).
\]
Many geometric equations are of the form \eqref{nonlinear equation}, such as complex Monge--Amp\`ere equation, complex Hessian equation, complex Hessian quotient equation and the Monge--Amp\`ere equation for $(n-1)$-plurisubharmonic functions.

For the complex Monge--Amp\`ere equation, when $(M,\omega)$ is K\"ahler, Yau \cite{Yau78} proved the existence of solution and solved Calabi's conjecture (see \cite{Calabi57}). When $(M,\omega)$ is Hermitian, the complex Monge--Amp\`{e}re equation has been solved under some assumptions on $\omega$ (see \cite{Cherrier87,Hanani96,GL10,TW10a,ZZ11}). The general Hermitian case was solved by Tosatti-Weinkove \cite{TW10b}. More generally, in the almost Hermitian setting, analogous result was obtained by Chu-Tosatti-Weinkove \cite{CTW19}.

When $(M,\omega)$ is K\"{a}hler, using a priori estimates of Hou \cite{Hou09} and Hou-Ma-Wu \cite{HMW10}, Dinew-Ko{\l}odziej \cite{DK17} solved the complex Hessian equation. This result was generalized to general Hermitian case by Sz\'ekelyhidi \cite{Szekelyhidi18} and Zhang \cite{Zhang17} independently.

For the complex Hessian quotient equation (i.e. complex $\sigma_{k}/\sigma_{l}$ equation where $1\leq l<k\leq n$), when $(M,\omega)$ is K\"ahler, $k=n$, $l=n-1$ and $h$ is constant, Song-Weinkove \cite{SW08} obtained the necessary and sufficient condition for the existence of solution. This condition is equivalent to the existence of $\mathcal{C}$-subsolution. This result was generalized by Fang-Lai-Ma \cite{FLM11} to $k=n$ and general $l$, and by Sz\'ekelyhidi \cite{Szekelyhidi18} to general $k$ and $l$. When $\omega$ is Hermitian and $h$ is not constant, analogous result was obtained by Sun \cite{Sun16,Sun17a,Sun17b} (see also \cite{Li14,GS15}). When $k=n$ and $1\leq l\leq n-1$, the third author \cite{ZJ21} proved a priori estimates and gave a sufficient condition for the existence of solution  in the almost Hermitian setting.

The Monge--Amp\`ere equation for $(n-1)$-plurisubharmonic functions was introduced and studied by Fu-Wang-Wu \cite{FWW10,FWW15}, which is a kind of Monge--Amp\`ere type equation as follows:
\begin{equation}\label{n-1}
\left(\eta+\frac{1}{n-1}\Big((\Delta^{C}u)\chi-\ddbar u\Big)\right)^{n} = e^{h}\chi^{n},
\end{equation}
where $\eta$ is a Hermitian metric, $\Delta^{C}$ denotes the canonical Laplacian operator of $\chi$. When $\chi$ is a K\"ahler metric, the equation \eqref{n-1} was solved by Tosatti-Weinkove \cite{TW17}. They later generalized this result to general Hermitian metric $\chi$ in \cite{TW19}.

Our main result is the following estimate:

\begin{theorem}\label{main estimate}
Let $(M,\chi,J)$ be a compact almost Hermitian manifold of real dimension $2n$ and $\underline{u}$ is a $\mathcal{C}$-subsolution (see Definition \ref{sub}) of \eqref{nonlinear equation}. Suppose that $u$ is a smooth solution of (\ref{nonlinear equation}) . Then for any $\alpha\in(0,1)$, we have the following estimate
\begin{equation*}
	\| u\|_{C^{2,\alpha}(M,\chi)}\leq C,
\end{equation*}
where $C$ is a constant depending only on  $\alpha$, $\underline{u}$, $h$, $\omega$, $f$, $\Gamma$ and $(M,\chi,J)$.
\end{theorem}

Using Theorem \ref{main estimate}, we can solve the complex Hessian equation and the Monge--Amp\`ere equation for $(n-1)$-plurisubharmonic equations on compact almost Hermitian manifolds. For the definitions of $k$-positivity and $\Gamma_{k}(M,\chi)$, we refer the reader to Section \ref{applications}.

\begin{theorem}\label{complex Hessian equation}
Let $(M,\chi,J)$ be a compact almost Hermitian manifold of real dimension $2n$ and $\omega$ be a smooth $k$-positive real $(1,1)$-form. For any integer $1\leq k\leq n$, there exists a unique pair $(u,c)\in C^{\infty}(M)\times\mathbb{R}$ such that
\begin{equation}\label{complex Hessian equation 1}
\begin{cases}
\ \omega_{u}^{k}\wedge \chi^{n-k}=e^{h+c}\chi^{n}, \\[1mm]
\ \omega_{u}\in\Gamma_{k}(M,\chi), \\[0.5mm]
\ \sup_{M}u = 0.
\end{cases}
\end{equation}
\end{theorem}

\begin{theorem}\label{n-1 MA equation}
Let $(M,\chi,J)$ be a compact almost Hermitian manifold of real dimension $2n$ and $\eta$ be an almost Hermitian metric. There exists a unique pair $(u,c)\in C^{\infty}(M)\times\mathbb{R}$ such that
\begin{equation}\label{n-1 MA equation 1}
\begin{cases}
\ \left(\eta+\frac{1}{n-1}\big((\Delta^{C} u)\chi-\ddbar u\big)\right)^{n} = e^{h+c}\chi^{n}, \\[2mm]
\ \eta+\frac{1}{n-1}\big((\Delta^{C} u)\chi-\ddbar u\big) > 0, \\[1.5mm]
\ \sup_{M}u = 0.
\end{cases}
\end{equation}
\end{theorem}

We now discuss the proof of Theorem \ref{main estimate}. The zero order estimate can be proved by adapting the arguments of \cite[Proposition 11]{Szekelyhidi18} and \cite[Proposition 3.1]{CTW19} which are based on the method of B\l ocki \cite{Blocki05,Blocki11}. For the second order estimate, we first show that the real Hessian can be controlled by the gradient quadratically as follows:
\begin{equation}\label{estimate introduction}
\sup_{M}|\nabla^{2}u|_{\chi} \leq C\sup_{M}|\partial u|_{\chi}^{2}+C.
\end{equation}
Then the second order estimate follows from the blowup argument and Liouville type theorem \cite[Theorem 20]{Szekelyhidi18}. To prove \eqref{estimate introduction}, one may follow the arguments of \cite[Proposition 5.1]{CTW19} and \cite[Theorem 4.1]{CM21}. However, these arguments do not seem to work for \eqref{nonlinear equation} in the almost Hermitian setting. Precisely, although \cite{CTW19} investigated the complex Monge--Amp\`ere equation on compact almost Hermitian manifolds, the argument of \cite[Proposition 5.1]{CTW19} depends heavily on the special structure of Monge--Amp\`ere operator. It is hard to generalize it to \eqref{nonlinear equation} directly. The estimate \eqref{estimate introduction} for the general equation \eqref{nonlinear equation} was considered in \cite{CM21}. But the underlying manifold in \cite{CM21} is only Hermitian, and the real Hessian estimate \cite[Theorem 4.1]{CM21} is built on the a priori estimates (including complex Hessian estimate) of \cite{Szekelyhidi18}. These make the situation considered in \cite{CM21} simpler than the almost Hermitian setting.

The proof of \eqref{estimate introduction} is the heart of this paper. We apply the maximum principle to a quantity involving the largest eigenvalue of $\nabla^{2}u$ as in \cite{CTW19,CM21}. The main task is to control negative third order terms by positive third order terms. We will use the ideas of \cite[Proposition 5.1]{CTW19} to deal with the terms arising from the non-integrability of the almost complex structure, and refine the techniques of  \cite[Theorem 4.1]{CM21} to control negative terms. More precisely, we modify the definitions of index sets $S$ and $I$ in \cite[Section 4.2]{CM21}, and introduce a new index set $J$. According to the index sets $J$ and $S$, we divide the proof of \eqref{estimate introduction} into three cases. The third case ($J,S\neq\emptyset$) is the most difficult case. We split the bad third order term $B$ (see \eqref{G3 B F}) into three terms $B_{1}$, $B_{2}$, $B_{3}$ (see \eqref{definition Bi}).  For the terms $B_{1}$ and $B_{2}$, since the constant $C$ in \eqref{estimate introduction} should be independent of $\sup_{M}|\partial u|_{\chi}^{2}$, we may not apply the argument of \cite[Section 4.2.1]{CM21} directly. Thanks to the definition of modified set $I$ and new set $J$, terms $B_{1}$ and $B_{2}$ can be controlled by some lower order terms (see Lemma \ref{bad terms 1 2}). For the main negative term $B_{3}$, we first derive a lower bound of $\omega_{u}$, which depends on $\sup_{M}|\de u|_{\chi}$ quadratically (see Lemma \ref{nu}). Combining this lower bound and some delicate calculations, we control $B_{3}$ by positive terms $G_{1}, G_{2}, G_{3}$ (see \eqref{G1 G2}, \eqref{G3 B F} and Lemma \ref{B 3}). Here $G_{2}$ comes from the concavity of equation \eqref{nonlinear equation}, and $G_{1}$, $G_{3}$ are provided by $\log\lambda_{1}$, $\xi(|\rho|^{2})$ in the quantity $Q$, respectively.

The paper is organized as follows. In Section \ref{preliminaries}, we will introduce some notations, recall definition and an important property of $\mathcal{C}$-subsolution. The zero order estimate will be established in Section \ref{zero order estimate}. We will derive key inequality \eqref{estimate introduction} in Section \ref{second order estimate}, and complete the proof of Theorem \ref{main estimate} in Section \ref{proof of main estimate}. In Section \ref{applications}, we will prove Theorem \ref{complex Hessian equation} and \ref{n-1 MA equation}.

\section{Preliminaries}\label{preliminaries}
\subsection{Notations}
Recall that $(M,\chi,J)$ is an almost Hermitian manifold of real dimension $2n$. Using the almost complex structure $J$, we can define $(p,q)$-form and operators $\de$, $\dbar$ (see e.g. \cite[p. 1954]{CTW19}). Write
\[
A^{1,1}(M) = \big\{ \alpha: \text{$\alpha$ is a smooth real (1,1)-forms on $(M,J)$} \big\}.
\]
For any $u\in C^{\infty}(M)$, we see that (see e.g. \cite[p. 1954]{CTW19})
\[
\ddbar u = \frac{1}{2}(dJdu)^{(1,1)}
\]
is a real $(1,1)$-form in $A^{1,1}(M)$.

For any point $x_{0}\in M$, let $\{e_{i}\}_{i=1}^{n}$ be a local unitary $(1,0)$-frame with respect to $\chi$ near $x_{0}$. Denote its dual coframe by $\{\theta^{i}\}_{i=1}^{n}$. Then we have
\[
\chi=\sqrt{-1}\delta_{ij}\theta^{i}\wedge\ov{\theta}^{j}.
\]
Suppose that
\[
\omega = \sqrt{-1}g_{i\ov{j}}\theta^{i}\wedge\ov{\theta}^{j}, \quad
\omega_{u} = \sqrt{-1}\ti{g}_{i\ov{j}}\theta^{i}\wedge\ov{\theta}^{j}.
\]
Then we have (see e.g. \cite[(2.5)]{HL15})
\[
\ti{g}_{i\ov{j}} = g_{i\ov{j}}+(\de\dbar u)(e_{i},\ov{e}_{j})
= g_{i\ov{j}}+e_{i}\overline{e}_{j}(u)-[e_{i},\overline{e}_{j}]^{(0,1)}(u),
\]
where $[e_{i}, \bar{e}_{j}]^{(0,1)}$ is the $(0,1)$-part of the Lie bracket $[e_{i}, \bar{e}_{j}]$. Define
\begin{equation*}
F^{i\overline{j}}=\frac{\partial F}{\de\ti{g}_{i\ov{j}}}, \quad
F^{i\overline{j},k\overline{l}}=\frac{\partial^{2}F}{\de\ti{g}_{i\ov{j}}\de\ti{g}_{k\ov{l}}}.
\end{equation*}
After making a unitary transformation, we may assume that  $\tilde{g}_{i\overline{j}}(x_{0})=\delta_{ij}\tilde{g}_{i\overline{i}}(x_{0})$.
We denote $\tilde{g}_{i\overline{i}}(x_{0})$ by $\mu_{i}$.
It is useful to order $\{\mu_{i}\}$ such that
\begin{equation}\label{mu order}
\mu_{1}\geq\mu_{2}\geq\cdots\geq\mu_{n}.
\end{equation}
At $x_{0}$, we have the expressions of $F^{i\ov{j}}$ and $F^{i\bar{j},k\bar{l}}$ (see e.g. \cite{Andrews94,Gerhardt96,Spruck05})
\begin{equation}\label{second derive of F}
F^{i\ov{j}} = \delta_{ij}f_{i},\quad
F^{i\bar{j},k\bar{l}}=f_{ik}\delta_{ij}\delta_{kl}
+\frac{f_{i}-f_{j}}{\mu_{i}-\mu_{j}}
(1-\delta_{ij})\delta_{il}\delta_{jk},
\end{equation}
where the quotient is interpreted as a limit if $\mu_{i}=\mu_{j}$. Using \eqref{mu order}, we obtain (see e.g. \cite{EH89,Spruck05})
\begin{equation}\label{F ii 1}
F^{1\ov{1}} \leq F^{2\ov{2}} \leq \cdots \leq F^{n\ov{n}}.
\end{equation}

On the other hand, the linearized operator of equation \eqref{nonlinear equation} is given by
\begin{equation}\label{L}
L:=\sum_{i,j}F^{i\bar{j}}(e_{i}\bar{e}_{j}-[e_{i},\bar{e}_{j}]^{0,1}).
\end{equation}
Note that $[e_{i},\bar{e}_{j}]^{(0,1)}$ is a first order differential operator, and so $L$ is a second order elliptic operator.

\subsection{$\mathcal{C}$-subsolution}
\begin{defn}[Definition 1 of \cite{Szekelyhidi18}]\label{sub}
We say that a smooth function $\underline{u}:M\rightarrow\mathbb{R}$ is a $\mathcal{C}$-subsolution of \eqref{nonlinear equation} if at each point $x\in M$, the set
\begin{equation*}
\left\{\mu\in \Gamma: f(\mu)=h(x) \text{\ and\ }  \mu-\mu(\underline{u})\in \Gamma_{n}\right\}
\end{equation*}
is bounded, where $\mu(\underline{u})=(\mu_{1}(\underline{u}),\ldots,\mu_{n}(\underline{u}))$ denote the eigenvalues of $\omega+\ddbar\underline u$ with respect to $\chi$.
\end{defn}
For $\sigma\in(\sup_{\de\Gamma}f, \sup_{\Gamma}f)$, we denote
\begin{equation}\label{level set}
	\Gamma^{\sigma}=\{\mu\in \Gamma: f(\mu)> \sigma\}.
\end{equation}
By Definition \ref{sub}, for any $\mathcal{C}$-subsolution $\underline{u}$, there are constants $\delta, R>0$ depending only on $\underline{u}$, $(M,\chi,J)$, $f$ and $\Gamma$ such that at each $x\in M$ we have
\begin{equation}\label{2..12}
	(\mu(\underline{u})-\delta\textbf{1}+\Gamma_{n})\cap \partial\Gamma^{h(x)}\subset B_{R}(0),
\end{equation}
where $\textbf{1}=(1,\ldots,1)\in\mathbb{R}^{n}$ and $B_{R}(0)\subset\mathbb{R}^{n}$ denotes the Euclidean ball with radius $R$ and center $0$.

The following proposition follows from \cite[Lemma 9]{Szekelyhidi18} and \cite[Proposition 6]{Szekelyhidi18} (which is a refinement of \cite[Theorem 2.18]{Guan14}). It will play an important role in the proof of Theorem \ref{Thm4.1}.

\begin{proposition}\label{prop subsolution}
Let $\sigma\in[\inf_{M}h,\sup_{M}h]$ and $A$ be a Hermitian matrix with eigenvalues $\mu(A)\in \partial \Gamma^{\sigma}$.
\begin{enumerate}\setlength{\itemsep}{1mm}
\item There exists a constant $\tau$ depending only on $f$, $\Gamma$ and $\sigma$ such that
\[
\mathcal{F}(A) := \sum_{i}F^{i\ol{i}}(A) > \tau.
\]
\item For $\delta, R>0$, there exists $\theta>0$ depending only on $f$, $\Gamma$, $h$, $\delta$, $R$ such that the following holds. If $B$ is a Hermitian matrix satisfying
\[
(\mu(B)-2\delta{\bf 1}+\Gamma_{n})\cap\partial\Gamma^{\sigma}\subset B_{R}(0),
\]
then we  have either
\[
\sum_{p,q}F^{p\ol{q}}(A)[B_{p\ol{q}}-A_{p\ol{q}}]>\theta\sum_{p}F^{p\ol{p}}(A)
\]
or
\[
F^{i\ol{i}}(A)>\theta\sum_{p}F^{p\ol{p}}(A), \quad \text{for all $i$}.
\]
\end{enumerate}
\end{proposition}

\begin{proof}
(1) is \cite[Lemma 9 (b)]{Szekelyhidi18}. For (2), when $|\mu(A)|>R$, the conclusion follows from \cite[Proposition 6]{Szekelyhidi18}. When $|\mu(A)|\leq R$, we consider the set
\[
S_{R,\sigma} = \left\{N\in \mathcal{H}_{n}:\mu(N)\in \ov{B_{R}(0)}\cap\partial \Gamma^{\sigma}\right\},
\]
where $\mathcal{H}_{n}$ denotes the set of Hermitian matrices of complex entries of size $n$. It is clear that $S_{R,\sigma}$ is compact, and then there exists a constant $C>0$ such that for $A\in S_{R,\sigma}$,
\[
C^{-1} \leq F^{i\ov{i}}(A) \leq C, \quad \text{for all $i$}.
\]
Decreasing $\theta$ if necessary,
\[
F^{i\ol{i}}(A)>\theta\sum_{p}F^{p\ol{p}}(A), \quad \text{for all $i$}.
\]
\end{proof}

\begin{remark}
In the proof of Proposition \ref{prop subsolution} (2), we use the compactness of the set $S_{R,\sigma}$. Then the constant $\theta$ depends on $\inf_{M}h$, and so Proposition \ref{prop subsolution} is only applicable in the non-degenerate setting. For the analogous result in the degenerate setting, we refer the reader to \cite[Section 3]{CM21}.
\end{remark}

\section{Zero order estimate}\label{zero order estimate}
The proof of zero order estimate will be given in this section. We follow the arguments of \cite[Proposition 11]{Szekelyhidi18} and \cite[Proposition 3.1]{CTW19}, which are generalizations of the arguments in \cite{Blocki05,Blocki11}. We begin with the following proposition.

\begin{proposition}[Proposition 2.3 of \cite{CTW19}]\label{L1}
Let $(M,\chi,J)$ be a compact almost Hermitian manifold. Suppose that $\vp$ satisfies
\begin{equation*}
\sup_{M}\vp = 0, \quad \Delta^{C}\vp \geq -B
\end{equation*}
for some constant $B$, where $\Delta^{C}$ denotes the canonical Laplacian operator of $\chi$. Then there exists a constant $C$ depending only on $B$ and $(M, \chi, J)$ such that
\begin{equation*}
\int_{M}(-\vp)\chi^{n}\leq C.
\end{equation*}
\end{proposition}

\begin{proof}
Actually, Proposition \ref{L1} is slightly stronger than \cite[Proposition 2.3]{CTW19}. Precisely, the assumption $\chi+\ddbar\vp>0$ in \cite[Proposition 2.3]{CTW19} is replaced by $\Delta^{C}\vp\geq-B$. Fortunately, as the reader may check, under this weak assumption, the argument of \cite[Proposition 2.3]{CTW19} still works.
\end{proof}

\begin{proposition}\label{Prop3.2}
Suppose that $\underline{u}$ is a $\mathcal{C}$-subsolution of \eqref{nonlinear equation}. Let $u$ be a smooth solution of \eqref{nonlinear equation} with $\sup_{M}(u-\underline{u})=0$. Then there exists a constant $C$ depending on $\|\underline{u}\|_{C^{2}}$, $\|h\|_{C^{0}}$, $\|\omega\|_{C^{0}}$, $f$, $\Gamma$ and $(M,\chi,J)$ such that
\begin{equation*}
\|u\|_{L^{\infty}}\leq C.
\end{equation*}
\end{proposition}

\begin{proof}
Replacing $\omega$ by $\omega_{\underline{u}}$, we may assume that $\underline{u}=0$. It then suffices to establish the lower bound of the infimum $I=\inf_{M}u$. Assume that $I$ is attained at $x_{0}$. Choose a local coordinate chart $(x^{1},\ldots,x^{2n})$ in a neighborhood of $x_{0}$ containing the unit ball $B_{1}(0)\subset \mathbb{R}^{2n}$ such that the point $x_{0}$ corresponds the origin $0\in \mathbb{R}^{2n}$.
	
Consider the function
\[
v= u+\ve\sum_{i=1}^{2n}(x^{i})^{2}
\]
for a small $\ve>0$ determined later. It is clear that
\[
v(0) = I, \quad
v\geq I+\ve \ \text{on $\partial B_{1}(0)$}.
\]
We define the set $P$ by
\begin{equation}\label{P}
P=\Big\{x\in B_{1}(0) : |Dv(x)|\leq \frac{\ve}{2},
v(y)\geq v(x)+Dv(x)\cdot (y-x),\ \forall\  y\in B_{1}(0)\Big\}.
\end{equation}
Thanks to \cite[Proposition 10]{Szekelyhidi18}, we have
\begin{equation}\label{3.5}
c_{0}\ve^{2n}\leq \int_{P}\det(D^{2}v),
\end{equation}
where $c_{0}$ is a constant depending only on $n$.

Let $(D^{2}u)^{J}$ be the $J$-invariant part of $D^{2}u$, i.e.
\begin{equation}\label{zero order D2 2}
(D^{2}u)^{J}=\frac{1}{2}(D^{2}u+J^{T}\cdot D^{2}u\cdot J),
\end{equation}
where $J^{T}$ denotes the transpose of $J$. The definition of $P$ \eqref{P} shows $D^{2}v\geq 0$ on $P$. Then
\begin{equation}\label{zero estimate D2}
(D^{2}u)^{J}(x)\geq (D^{2}v)^{J}(x)-C\ve\,\mathrm{Id}\geq -C\ve\,\mathrm{Id} \ \ \text{on $P$}.
\end{equation}
Consider the bilinear form $H(u)(X,Y)=\ddbar u(X, JY)$ on $B_{1}(0)$. Direct calculation shows (see e.g. \cite[p. 443]{TWWY15})
\[
H(u)(X,Y)(x)=\frac{1}{2}(D^{2}u)^{J}(x)+E(u)(x),
\]	
where $E(u)(x)$ is an error matrix which depends linearly on $Du(x)$. On $P$, we have $|Dv|\leq\frac{\ve}{2}$ and so $|Du|\leq \frac{5\ve}{2}$. Combining this with \eqref{zero estimate D2},
\[
H(u) \geq -C\ve\,\mathrm{Id}.
\]
It then follows that
\[
\omega_{u}-\omega = \ddbar u \geq -C\ve\chi.
\]
Choosing $\ve$ sufficiently small such that $C\ve\leq \delta$, we see that
\[
\mu(u)\in \mu(0)-\delta \textbf{1}+\Gamma_{n}.
\]
Since $u$ solves the equation \eqref{nonlinear equation}, then $\mu(u)\in \partial\Gamma^{h}$. By \eqref{2..12},
\[
\mu(u)\in (\mu(0)-\delta \textbf{1}+\Gamma_{n})\cap \partial\Gamma^{h} \subset B_{R}(0)
\]
for some $R>0$. This gives upper bound for $H(u)$. We define $H(v)$ similarly and obtain
\begin{equation}\label{upper bound of D 2 v J}
(D^{2}v)^{J} = 2H(v)-2E(v) \leq 2H(u)+C\ve\,\mathrm{Id} \leq C\,\mathrm{Id}.
\end{equation}
Note that $\det(A+B)\geq \det(A)+\det(B)$ for non-negative definite Hermitian matrices $A,B$. On $P$, we have $D^{2}v\geq 0$ and so
\[
\det((D^{2}v)^{J}) = 2^{-2n}\det(D^{2}v+J^{T}\cdot D^{2}v\cdot J)
\geq 2^{-2n+1}\det(D^{2}v).
\]
Combining this with \eqref{upper bound of D 2 v J},
\[
\det(D^{2}v) \leq 2^{2n-1}\det((D^{2}v)^{J})\leq C.
\]
Plugging this into (\ref{3.5}), we obtain
\begin{equation}\label{zero estimates 3}
c_{0}\ve^{2n}\leq C|P|.
\end{equation}
For each $x\in P$, choosing $y=0$ in $\eqref{P}$, we have
\begin{equation*}
I=v(0)\geq v(x)-|Dv(x)|\cdot|x|\geq v(x)-\frac{\ve}{2}.
\end{equation*}
We assume without loss of generality that $I+\ve\leq 0$. It then follows that
\[
|I+\ve| \leq -v \ \ \text{on $P$}.
\]
Integrating both sides and using \eqref{zero estimates 3}, we obtain
\[
c_{0}\ve^{2n} \leq |P| \leq \frac{\int_{P}(-v)\chi^{n}}{|I+\ve|}.
\]
On the other hand, the assumptions of $\Gamma$ shows (see \cite[(4)]{CNS85} or \cite[(44)]{Szekelyhidi18})
\[
\Gamma \subset \Gamma_{1} := \big\{(w_{1},\ldots,w_{n}):\sum_{i}w_{i}>0\big\}.
\]
Then we have $\mathrm{tr}_{\chi}\omega_{u}>0$ and so
\begin{equation}\label{Delta u lower bound}
\Delta^{C}u = \mathrm{tr}_{\chi}\omega_{u}-\mathrm{tr}_{\chi}\omega \geq -C.
\end{equation}
Using Proposition \ref{L1},
\[
c_{0}\ve^{2n}|I+\ve| \leq \int_{P}(-v)\chi^{n} \leq \int_{P}(-u)\chi^{n}+C\ve \leq C.
\]
This gives the required estimate of $I$.
\end{proof}

\section{Second order estimate}\label{second order estimate}
In this section, we prove the following second order estimate.
\begin{theorem}\label{Thm4.1}
Suppose that $\underline{u}$ is a $\mathcal{C}$-subsolution of \eqref{nonlinear equation}. Let $u$ be a smooth solution of \eqref{nonlinear equation} with $\sup_{M}(u-\underline{u})=0$. Then there exists a constant $C$ depending only on $\|\underline{u}\|_{C^{4}}$, $\|h\|_{C^{2}}$, $\|\omega\|_{C^{2}}$, $f$, $\Gamma$ and $(M,\chi,J)$ such that
\begin{equation}\label{aa}
\sup_{M}|\nabla^{2}u|_{\chi} \leq C\sup_{M}|\partial u|_{\chi}^{2}+C,
\end{equation}
where $\nabla$ denotes the Levi-Civita connection with respect to $\chi$.
\end{theorem}

Replacing $\omega$ by $\omega_{\underline{u}}$ and $u$ by $u-1$, we may assume $\underline{u}=0$ and $\sup_{M}u=-1$. Let $\lambda_{1}\geq\lambda_{2}\geq\cdots\geq\lambda_{2n}$ be the eigenvalues of $\nabla^{2}u$ with respect to $\chi$. For notational convenience, we write
\[
|\cdot| = |\cdot|_{\chi}.
\]
Combining $\Delta^{C}u\geq-C$ (see \eqref{Delta u lower bound}) with \cite[(5.2)]{CTW19},
\[
\sum_{\alpha=1}^{2n}\lambda_{\alpha} \geq 2\Delta^{C}u-C\sup_{M}|\de u|
\geq -C\sup_{M}|\de u|-C,
\]
which implies
\begin{equation}\label{Hessian lambda 1}
|\nabla^{2}u| \leq C\lambda_{1}+C\sup_{M}|\de u|+C.
\end{equation}

On the other hand, setting
\begin{equation*}
K=\sup_{M}|\partial u|^{2}+1, \quad  N=\sup_{M} |\nabla^{2} u|+1, \quad \rho=\nabla^{2}u+N\chi.
\end{equation*}
Let us remark that $\rho>0$ by definition. We consider the quantity
\begin{equation*}
Q = \log \lambda_{1}+\xi(|\rho|^{2})+\eta(|\partial u|^{2})+e^{-Au}
\end{equation*}
on $\Omega:=\{\lambda_{1}>0\}\subset M$, where
\[
\xi(s) = -\frac{1}{4}\log(5N^{2}-s), \quad
\eta(s) = -\frac{1}{4}\log(2K-s)
\]
and $A>0$ is a large constant to be determined later. It is clear that
\begin{equation}\label{xieta}
\begin{split}
\xi''=4(\xi')^{2},& \quad \frac{1}{20N^{2}} \leq \xi' \leq \frac{1}{4N^{2}},\\[1mm]
\eta''=4(\eta')^{2}, & \quad  \frac{1}{8K} \leq \eta' \leq \frac{1}{4K}.
\end{split}
\end{equation}
We may assume $\Omega$ is a nonempty open set (otherwise we are done by \eqref{Hessian lambda 1}). Let $x_{0}$ be a maximum point of $Q$. Note that $Q(z)\ri -\infty$ as $z$ approaches to $\partial{\Omega}$. Then $x_{0}$ must be an interior point of $\Omega$. To prove Theorem \ref{Thm4.1}, it suffices to show
\begin{equation}\label{goal}
\lambda_{1}(x_{0}) \leq CK.
\end{equation}
Indeed, by the definition of $Q$ and Proposition \ref{Prop3.2},
\begin{equation}\label{Q x0}
Q(x_{0}) \leq \log\lambda_{1}(x_{0})-\frac{1}{4}\log N^{2}-\frac{1}{4}\log K+e^{CA}.
\end{equation}
Let $y_{0}$ be the maximum point of $|\nabla^{2}u|$. Then \eqref{Hessian lambda 1} implies
\[
\sup_{M} |\nabla^{2} u|+1 = N \leq C\lambda_{1}(y_{0})+C\sup_{M}|\de u|+C
\leq C\lambda_{1}(y_{0})+C\sqrt{K}.
\]
Without loss of generality, we assume that $N\geq 2C\sqrt{K}$. Then $\lambda_{1}(y_{0})\geq\frac{N}{2C}>0$ and so
\begin{equation}\label{Q y0}
\begin{split}
Q(y_{0}) \geq {} & \log\lambda_{1}(y_{0})-\frac{1}{4}\log(5N^{2})-\frac{1}{4}\log(2K) \\
\geq {} & \log\frac{N}{2C}-\frac{1}{2}\log N-\frac{1}{4}\log K-C \\
\geq {} & \frac{1}{2}\log N-\frac{1}{4}\log K-C.
\end{split}
\end{equation}
Recalling that $x_{0}$ is the maximum point of $Q$, we have $Q(y_{0})\leq Q(x_{0})$. Combining this with \eqref{Q x0} and \eqref{Q y0},
\begin{equation}\label{N lambda 1}
N \leq C_{A}\lambda_{1}(x_{0})
\end{equation}
for some constant $C_{A}$ depending on $A$. This shows \eqref{goal} implies Theorem \ref{Thm4.1}.

Near $x_{0}$, we choose a local unitary frame $\{e_{i}\}_{i=1}^{n}$ with respect to $\chi$ such that at $x_{0}$,
\begin{equation}\label{5..6}
\chi_{i\overline{j}}=\delta_{ij}, \quad \tilde{g}_{i\overline{j}}=\delta_{ij}\tilde{g}_{i\overline{i}}, \quad \tilde{g}_{1\overline{1}}\geq\tilde{g}_{2\overline{2}}\geq\cdots\geq\tilde{g}_{n\overline{n}}.
\end{equation}
Note that $\ti{g}_{i\ov{j}}$ is defined by $\omega_{u}=\sqrt{-1}\ti{g}_{i\ov{j}}\theta^{i}\wedge\ov{\theta}^{j}$, where $\{\theta^{i}\}_{i=1}^{n}$ denotes the dual coframe of $\{e_{i}\}_{i=1}^{n}$. By \eqref{F ii 1}, at $x_{0}$, we have
\begin{equation}\label{F ii}
F^{1\ov{1}} \leq F^{2\ov{2}} \leq \cdots \leq F^{n\ov{n}}.
\end{equation}
In addition, since $(M,\chi,J)$ is almost Hermitian, then $\chi$ and $J$ are compatible, and so there exists a coordinate system $(U,\{x^{\alpha}\}_{\alpha=1}^{2n})$ in a neighborhood of $x_{0}$ such that at $x_{0}$,
\begin{equation}\label{real frame and complex frame}
e_{i}=\frac{1}{\sqrt{2}}\left(\frac{\de}{\de x^{2i-1}}-\sqrt{-1}\frac{\de}{\de x^{2i}}\right), \quad \text{for $i=1,2,\cdots,n$},
\end{equation}
and
\begin{equation}\label{first derivative of background metric}
\frac{\partial \chi_{\alpha\beta}}{\partial x^{\gamma}}=0, \quad \text{for $\alpha,\beta,\gamma=1,2,\cdots,2n$},
\end{equation}
where $\chi_{\alpha\beta}=\chi(\partial_{\alpha}, \partial_{\beta})$ and $\de_{\alpha}=\frac{\de}{\de x^{\alpha}}$. Write $u_{\alpha\beta}=(\nabla^{2}u)(\de_{\alpha},\de_{\beta})$ and define
\[
\Phi_{\beta}^{\alpha} := \chi^{\alpha\gamma}u_{\gamma\beta},
\]
where $\chi^{\alpha\gamma}$ denotes the inverse of the matrix $(\chi_{\alpha\gamma})$. Recall that $\lambda_{\alpha}$ denote the eigenvalues of $\nabla^{2}u$ with respect to $\chi$. Then $\lambda_{\alpha}$ are also the eigenvalues of $\Phi$. Let $V_{1},\ldots,V_{2n}$ be the eigenvectors for $\Phi$ at $x_{0}$, corresponding to eigenvalues $\lambda_{1},\ldots,\lambda_{2n}$ respectively. Define $V_{\alpha}^{\beta}$ by $V_{\alpha}=V_{\alpha}^{\beta}\de_{\beta}$ at $x_{0}$, and extend $V_{\alpha}$ to be vector fields near $x_{0}$ by taking the components to be constants. Using the perturbation argument (see \cite[p. 1965]{CTW19}), we may assume that $\lambda_{1}>\lambda_{2}$ at $x_{0}$, and so $\lambda_{1}$ is smooth near $x_{0}$.

At $x_{0}$, the maximum principle shows
\begin{equation}\label{5..10}
\frac{(\lambda_{1})_{i}}{\lambda_{1}}=-\xi'e_{i}(|\rho|^{2})-\eta'e_{i}(|\partial u|^{2})+Ae^{-Au}u_{i}, \quad
\text{for $i=1,2,\cdots,n$},
\end{equation}
and
\begin{equation}\label{L Q}
\begin{split}
0 & \geq  L(Q) \\[2mm]
&= \frac{L(\lambda_{1})}{\lambda_{1}}-\frac{F^{i\bar{i}}|(\lambda_{1})_{i}|^{2}}{\lambda_{1}^{2}}
+\xi' L(|\rho|^{2})+\xi''F^{i\bar{i}}|e_{i}(|\rho|^{2})|^{2}
 \\[1mm]
&+\eta'L(|\partial u|^{2})+\eta'' F^{i\bar{i}}|e_{i}(|\partial u|^{2})|^{2} -Ae^{-Au}L(u)+A^{2}e^{-Au} F^{i\bar{i}}|u_{i}|^{2},
\end{split}
\end{equation}
where the operator $L$ is defined in \eqref{L}.

From now on, all the calculations are done at $x_{0}$. We will use the Einstein summation convention. We usually use $C$ to denote a constant depending only on $\|h\|_{C^{2}}$, $\|\omega\|_{C^{2}}$, $f$, $\Gamma$, $(M,\chi,J)$, and $C_{A}$ to denote a constant depending only on $A$, $\|h\|_{C^{2}}$, $\|\omega\|_{C^{2}}$, $f$, $\Gamma$, $(M,\chi,J)$. In the following argument, we always assume without loss of generality that $\lambda_{1}\geq CK$ for some $C$, or $\lambda_{1}\geq C_{A}K$ for some $C_{A}$.

\subsection{Lower bound for $L(Q)$}
\begin{proposition}\label{lower bound of L Q}
For $\ve\in (0,\frac{1}{3}]$, at $x_{0}$, we have
\begin{equation}\label{4.7"}
\begin{split}
0  \geq L(Q) & \geq
G_{1}+G_{2}+G_{3}-B +\xi''F^{i\bar{i}}|e_{i}(|\rho|^{2})|^{2} \\[3mm]
& + \frac{3\eta'}{4}\sum_{i,j} F^{i\bar{i}}(|e_{i}e_{j}u|^{2}+|e_{i}\bar{e}_{j}u|^{2})
+\eta'' F^{i\bar{i}}|e_{i}(|\partial u|^{2})|^{2}\\&
-Ae^{-Au}L(u)+A^{2}e^{-Au} F^{i\bar{i}}|u_{i}|^{2}-\frac{C}{\ve}\mathcal{F},
\end{split}
\end{equation}
where
\begin{equation}\label{G1 G2}
G_1 := (2-\ve)\sum_{\alpha>1}\frac{F^{i\ov{i}}|e_{i}(u_{V_{1}V_{\alpha}})|^2}{\lambda_{1}(\lambda_{1}-\lambda_{\alpha})}, \quad
G_2:= -\frac{1}{\lambda_{1}} F^{i\bar{k},j\bar{l}}V_{1}(\tilde{g}_{i\bar{k}})V_{1}(\tilde{g}_{j\bar{l}}),
\end{equation}
\begin{equation}\label{G3 B F}
G_3 := \sum_{\alpha,\beta}\frac{F^{i\ov{i}}|e_{i}(u_{\alpha\beta})|^{2}}
{C_{A}\lambda_{1}^{2}} ,\quad
B := (1+\ve)\frac{F^{i\bar{i}}|(\lambda_{1})_{i}|^{2}}{\lambda_{1}^{2}}, \quad
\mathcal{F} := \sum_{i}F^{i\ov{i}}.
\end{equation}
\end{proposition}

To prove Proposition \ref{lower bound of L Q}, we need to estimate the lower bounds of $L(\lambda_{1})$, $L(|\rho|^{2})$ and $L(|\de u|^{2})$ respectively.

\subsubsection{Lower bound of $L(\lambda_{1})$}
\begin{lemma}\label{lower bound of L lambda1}
For $\ve\in(0,\frac{1}{3}]$, at $x_{0}$, we have
\[
L(\lambda_{1}) \geq
(2-\ve)\sum_{\alpha>1}\frac{F^{i\bar{i}}|e_{i}(u_{V_{1}V_{\alpha}})|^{2}}{\lambda_{1}-\lambda_{\alpha}}
-F^{i\bar{k},j\bar{l}}V_{1}(\tilde{g}_{i\bar{k}})V_{1}(\tilde{g}_{j\bar{l}})
-\ve\frac{F^{i\bar{i}}|(\lambda_{1})_{i}|^{2}}{\lambda_{1}}
-\frac{C}{\ve}\lambda_{1}\mathcal{F}.
\]
\end{lemma}

\begin{proof}
The following formulas are well-known (see e.g. \cite[Lemma 5.2]{CTW19}):
\begin{equation*}
\begin{split}
\frac{\partial \lambda_{1}}{\partial \Phi^{\alpha}_{\beta} }
= {} & V_{1}^{\alpha}V_{1}^{\beta}, \\
\frac{\partial^{2} \lambda_{1}}{\partial \Phi^{\alpha}_{\beta}\partial \Phi^{\gamma}_{\delta}}
= {} & \sum_{\mu>1}\frac{V_{1}^{\alpha}V_{\mu}^{\beta}V_{\mu}^{\gamma}V_{1}^{\delta} +V_{\mu}^{\alpha}V_{1}^{\beta}V_{1}^{\gamma}V_{\mu}^{\delta}}{\lambda_{1}-\lambda_{\mu}}.
\end{split}
\end{equation*}
Then we compute
\begin{equation}\label{lower bound of L lambda1 eqn 1}
\begin{split}
L(\lambda_{1})
= {} & F^{i\bar{i}}\frac{\partial^{2} \lambda_{1}}{\partial \Phi^{\alpha}_{\beta}\partial \Phi^{\gamma}_{\delta}}e_{i}(\Phi^{\gamma}_{\delta})\bar{e}_{i}(\Phi^{\alpha}_{\beta})
+F^{i\bar{i}}\frac{\partial \lambda_{1}}{\partial\Phi^{\alpha}_{\beta}}(e_{i}\bar{e}_{i}
-[e_{i},\bar{e}_{i}]^{(0,1)})(\Phi^{\alpha}_{\beta})\\
= {} & F^{i\bar{i}}\frac{\partial^{2} \lambda_{1}}{\partial \Phi^{\alpha}_{\beta}\partial \Phi^{\gamma}_{\delta}}e_{i}(u_{\gamma\delta})\bar{e}_{i}(u_{\alpha\beta})+F^{i\bar{i}}\frac{\partial \lambda_{1}}{\partial\Phi^{\alpha}_{\beta}}(e_{i}\bar{e}_{i}-[e_{i},\bar{e}_{i}]^{(0,1)})(u_{\alpha\beta})\\
&+F^{i\bar{i}}\frac{\partial \lambda_{1}}{\partial\Phi^{\alpha}_{\beta}}u_{\gamma\beta}e_{i}\bar{e}_{i}(\chi^{\alpha\gamma})\\
\geq {} & 2\sum_{\alpha>1}\frac{F^{i\bar{i}}|e_{i}(u_{V_{1}V_{\alpha}})|^{2}}{\lambda_{1}-\lambda_{\alpha}}
+F^{i\bar{i}}(e_{i}\bar{e}_{i}-[e_{i},\bar{e}_{i}]^{(0,1)})(u_{V_{1}V_{1}})-C\lambda_{1}\mathcal{F},\\
\end{split}
\end{equation}
where $(\chi^{\alpha\beta})$ denotes the inverse of the matrix $(\chi_{\alpha\beta})$. Note that $\mathcal{F}\geq \tau$ has uniform positive lower bound thanks to Proposition \ref{prop subsolution}.

\begin{claim}\label{claim 1}
At $x_{0}$, we have
\begin{equation*}
\begin{split}
F^{i\bar{i}}(e_{i}\bar{e}_{i}-[e_{i},\bar{e}_{i}]^{(0,1)})(u_{V_{1}V_{1}})
\geq {}  -F^{i\bar{k},j\bar{l}}V_{1}(\tilde{g}_{i\bar{k}})V_{1}(\tilde{g}_{j\bar{l}})
-C\lambda_{1}\mathcal{F}-2T,
\end{split}
\end{equation*}
where
\[
T:= F^{i\bar{i}}\big\{[V_{1},\bar{e}_{i}]V_{1}e_{i}(u)+[V_{1},e_{i}]V_{1}\bar{e}_{i}(u)\big\}.
\]
\end{claim}

\begin{proof}[Proof of Claim \ref{claim 1}]
It is clear that
\begin{equation}\label{claim2.1}
\begin{split}
& F^{i\bar{i}}(e_{i}\bar{e}_{i}-[e_{i},\bar{e}_{i}]^{(0,1)})(u_{V_{1}V_{1}}) \\
= {} & F^{i\bar{i}}(e_{i}\bar{e}_{i}-[e_{i},\bar{e}_{i}]^{(0,1)})
(V_{1}V_{1}(u)-(\nabla_{V_{1}}V_{1})(u)) \\
\geq {} &
F^{i\bar{i}}e_{i}\bar{e}_{i}V_{1}V_{1}(u)-F^{i\bar{i}}e_{i}\bar{e}_{i}(\nabla_{V_{1}}V_{1})(u)
-F^{i\bar{i}}[e_{i},\bar{e}_{i}]^{(0,1)}V_{1}V_{1}(u)-C\lambda_{1}\mathcal{F}.
\end{split}
\end{equation}	
Set $W=\nabla_{V_{1}}V_{1}$. Then
\begin{equation}\label{claim 1 eqn 4}
\begin{split}
& e_{i}\bar{e}_{i}W(u)
= e_{i}W\bar{e}_{i}(u)+e_{i}[\bar{e}_{i},W](u)\\[1mm]
= {} & We_{i}\bar{e}_{i}(u)+[e_{i},W]\bar{e}_{i}(u)+e_{i}[\bar{e}_{i},W](u)\\
= {} & W(\tilde{g}_{i\bar{i}}){-W(g_{i\ov{i}})}+W[e_{i},\bar{e}_{i}]^{(0,1)}(u)
+[e_{i},W]\bar{e}_{i}(u)+e_{i}[\bar{e}_{i},W](u).
\end{split}
\end{equation}
Applying $W$ to the equation \eqref{nonlinear equation},
\begin{equation*}
F^{i\bar{i}}W(\tilde{g}_{i\bar{i}})=W(h).
\end{equation*}
Substituting this into \eqref{claim 1 eqn 4},
\[
|F^{i\bar{i}}e_{i}\bar{e}_{i}(\nabla_{V_{1}}V_{1})(u)|
= |F^{i\bar{i}}e_{i}\bar{e}_{i}W(u)|
\leq C\lambda_{1}\mathcal{F}.
\]
Combining this with \eqref{claim2.1},
\begin{equation}\label{claim 1 eqn 1}
\begin{split}
& F^{i\bar{i}}(e_{i}\bar{e}_{i}-[e_{i},\bar{e}_{i}]^{(0,1)})(u_{V_{1}V_{1}}) \\
\geq {} & F^{i\bar{i}}\big\{e_{i}\bar{e}_{i}V_{1}V_{1}(u)
-[e_{i},\bar{e}_{i}]^{(0,1)}V_{1}V_{1}(u)\big\}-C\lambda_{1}\mathcal{F}.
\end{split}
\end{equation}
We use $O(\lambda_{1})$ to denote a term satisfying $|O(\lambda_{1})|\leq C\lambda_{1}$ for some uniform constant $C$. Then direct calculation shows
\begin{equation}\label{claim 1 eqn 5}
\begin{split}
& F^{i\bar{i}}\big\{e_{i}\bar{e}_{i}V_{1}V_{1}(u)
-[e_{i},\bar{e}_{i}]^{(0,1)}V_{1}V_{1}(u)\big\}\\
= {} & F^{i\bar{i}}\big\{e_{i}V_{1}\bar{e}_{i}V_{1}(u)-e_{i}[V_{1},\bar{e}_{i}]V_{1}(u)
-V_{1}[e_{i},\bar{e}_{i}]^{(0,1)}V_{1}(u)\big\}+O(\lambda_{1})\mathcal{F} \\
= {} & F^{i\bar{i}}\big\{V_{1}e_{i}\bar{e}_{i}V_{1}(u)-[V_{1},e_{i}]\bar{e}_{i}V_{1}(u)-[V_{1},\bar{e}_{i}]e_{i}V_{1}(u)
-V_{1}V_{1}[e_{i},\bar{e}_{i}]^{(0,1)}(u)\big\}+O(\lambda_{1})\mathcal{F}\\	
= {} & F^{i\bar{i}}\big\{V_{1}e_{i}\bar{e}_{i}V_{1}(u)-V_{1}V_{1}[e_{i},\bar{e}_{i}]^{(0,1)}(u)\big\}
+O(\lambda_{1})\mathcal{F}-T\\
= {} & F^{i\bar{i}}\big\{V_{1}e_{i}V_{1}\bar{e}_{i}(u)-V_{1}e_{i}[V_{1},\bar{e}_{i}](u)-V_{1}V_{1}[e_{i},\bar{e}_{i}]^{(0,1)}(u)\big\}
+O(\lambda_{1})\mathcal{F}-T\\
= {} & F^{i\bar{i}}\big\{V_{1}V_{1}(e_{i}\bar{e}_{i}-[e_{i},\bar{e}_{i}]^{(0,1)})(u)-V_{1}[V_{1},e_{i}]\bar{e}_{i}(u)-V_{1}e_{i}[V_{1},\bar{e}_{i}](u)\big\}
+O(\lambda_{1})\mathcal{F}-T\\
= {} & F^{i\bar{i}}V_{1}V_{1}\big(e_{i}\bar{e}_{i}(u)-[e_{i},\bar{e}_{i}]^{(0,1)}(u)\big)
+O(\lambda_{1})\mathcal{F}-2T \\
= {} & F^{i\bar{i}}V_{1}V_{1}(\ti{g}_{i\ov{i}})+O(\lambda_{1})\mathcal{F}-2T,
\end{split}
\end{equation}
where in the second-to-last line we used
\[
\begin{split}
& F^{i\bar{i}}(V_{1}[V_{1},e_{i}]\bar{e}_{i}(u)+V_{1}e_{i}[V_{1},\bar{e}_{i}](u)) \\
= {} & T+F^{i\bar{i}}\{[V_{1},[V_{1},e_{i}]]\bar{e}_{i}(u)+[e_{i},[V_{1},\bar{e}_{i}]](u)+[V_{1},[V_{1},\bar{e}_{i}]]e_{i}(u)\} \\
= {} & T+O(\lambda_{1})\mathcal{F}.
\end{split}
\]
Substituting \eqref{claim 1 eqn 5} into \eqref{claim 1 eqn 1},
\begin{equation}\label{claim 1 eqn 2}
F^{i\bar{i}}(e_{i}\bar{e}_{i}-[e_{i},\bar{e}_{i}]^{(0,1)})(u_{V_{1}V_{1}})
\geq F^{i\bar{i}}V_{1}V_{1}(\ti{g}_{i\ov{i}})-C\lambda_{1}\mathcal{F}-2T.
\end{equation}
To deal with the first term, we apply $V_{1}V_{1}$ to the equation \eqref{nonlinear equation} and obtain
\begin{equation}\label{claim 1 eqn 3}
F^{i\bar{i}}V_{1}V_{1}(\tilde{g}_{i\bar{i}})
= -F^{i\bar{k},j\bar{l}}V_{1}(\tilde{g}_{i\bar{k}})V_{1}(\tilde{g}_{j\bar{l}})+V_{1}V_{1}(h).
\end{equation}
Then Claim 1 follows from \eqref{claim 1 eqn 2} and \eqref{claim 1 eqn 3}.
\end{proof}

\begin{claim}\label{claim 2}
For $\ve\in(0,\frac{1}{3}]$, at $x_{0}$, we have
\[
2T \leq \ve\frac{F^{i\bar{i}}|(\lambda_{1})_{i}|^{2}}{\lambda_{1}}
+\ve\sum_{\alpha>1}\frac{F^{i\bar{i}}|e_{i}(u_{V_{1}V_{\alpha}})|^{2}}{\lambda_{1}-\lambda_{\alpha}}
+\frac{C}{\ve}\lambda_{1}\mathcal{F},
\]
assuming without loss of generality that $\lambda_{1}\geq K$.
\end{claim}

\begin{proof}[Proof of Claim \ref{claim 2}]
At $x_{0}$, we write
\[
[V_{1},e_{i}]=\sum_{\beta}\tau_{i\beta}V_{\beta}, \quad
[V_{1},\bar{e}_{i}]=\sum_{\beta}\overline{\tau_{i\beta}}V_{\beta},
\]
where $\tau_{i\beta}\in\mathbb{C}$ are uniformly bounded constants. Then
\[
2T = 2F^{i\bar{i}}\big\{[V_{1},\bar{e}_{i}]V_{1}e_{i}(u)
+[V_{1},e_{i}]V_{1}\bar{e}_{i}(u)\big\}
\leq C\sum_{\alpha}F^{i\ov{i}}|V_{\alpha}V_{1}e_{i}(u)|.
\]
Using $\lambda_{1}\geq K$, we compute
\[
\begin{split}
\big|V_{\alpha}V_{1}e_{i}(u)\big|
= {} &
\big|e_{i}V_{\alpha}V_{1}(u)+V_{\alpha}[V_{1},e_{i}](u)+[V_{\alpha},e_{i}]V_{1}(u)\big|  \\
= {} & \big|e_{i}(u_{V_{\alpha}V_{1}})+e_{i}(\nabla_{V_{1}}V_{\alpha})(u)+V_{\alpha}[V_{1},e_{i}](u)
+[V_{\alpha},e_{i}]V_{1}(u)\big|\\
\leq {} & \big|e_{i}(u_{V_{1}V_{\alpha}})\big|+C\lambda_{1}.
\end{split}
\]
It then follows that
\[
\begin{split}
2T \leq {} &
C\sum_{\alpha}F^{i\ov{i}}\big|e_{i}(u_{V_{1}V_{\alpha}})\big|+C\lambda_{1}\mathcal{F}\\
= {} & CF^{i\ov{i}}\big|e_{i}(u_{V_{1}V_{1}})\big| +C\sum_{\alpha>1}F^{i\ov{i}}\big|e_{i}(u_{V_{1}V_{\alpha}})\big|+C\lambda_{1}\mathcal{F}\\
\leq {} & \ve\frac{F^{i\bar{i}}|e_{i}(u_{V_{1}V_{1}})|^{2}}{\lambda_{1}}
+\frac{C}{\ve}\lambda_{1}\mathcal{F}
+\ve\sum_{\alpha>1}\frac{F^{i\bar{i}}|e_{i}(u_{V_{1}V_{\alpha}})|^{2}}{\lambda_{1}-\lambda_{\alpha}}	+\frac{C}{\ve}\sum_{\alpha>1}(\lambda_{1}-\lambda_{\alpha})\mathcal{F}+C\lambda_{1}\mathcal{F}\\
\leq {} & \ve\frac{F^{i\bar{i}}|(\lambda_{1})_{i}|^{2}}{\lambda_{1}}
+\ve\sum_{\alpha>1}\frac{F^{i\bar{i}}|e_{i}(u_{V_{1}V_{\alpha}})|^{2}}{\lambda_{1}-\lambda_{\alpha}}+\frac{C}{\ve}\lambda_{1}\mathcal{F},
\end{split}
\]
where we used $(\lambda_{1})_{i}=e_{i}(u_{V_{1}V_{1}})$ and \eqref{Hessian lambda 1} in the last inequality.
\end{proof}

Combining \eqref{lower bound of L lambda1 eqn 1}, Claim 1 and 2, we obtain Lemma \ref{lower bound of L lambda1}.
\end{proof}

\subsubsection{Lower bound of $L(|\rho|^{2})$}
\begin{lemma}\label{lower bound of L rho}
For $\ve\in(0,\frac{1}{3}]$, at $x_{0}$, we have
\begin{equation*}
\begin{split}
L(|\rho|^{2})\geq (2-\varepsilon)\sum_{\alpha,\beta}F^{i\bar{i}}|e_{i}(u_{\alpha\beta})|^{2}-\frac{C}{\varepsilon}N^{2}\mathcal F.\\
\end{split}
\end{equation*}
\end{lemma}

\begin{proof}
Applying $\de_{\alpha}\de_{\beta}$ to the equation \eqref{nonlinear equation},
\begin{equation*}
F^{i\bar{i}}\partial_{\alpha}\partial_{\beta}(\tilde{g}_{i\bar{i}})=- F^{i\bar{k},j\bar{l}}\partial_{\alpha}(\tilde{g}_{i\bar{k}})\partial_{\beta}(\tilde{g}_{j\bar{l}})+\partial_{\alpha}\partial_{\beta}(h).	\end{equation*}
By the similar calculation of \eqref{claim 1 eqn 2}, we obtain
\begin{equation*}
\begin{split}
& F^{i\bar{i}}(e_{i}\bar e_{i}-[e_{i}, \bar{e}_{i}]^{(0,1)})(u_{\alpha\beta})\\
= {} & F^{i\bar{i}}\partial_{\alpha}\partial_{\beta}(\ti{g}_{i\bar{i}})
-2F^{i\bar{i}}\mathrm{Re}([\partial_{\beta}, e_{i}]\partial_{\alpha}\bar{e}_{i}u)-2F^{i\bar{i}} \mathrm{Re}([\partial_{\alpha}, e_{i}] \partial_{\beta}\bar{e}_{i}u)+O(\lambda_{1})\mathcal{F}\\
= {} & -F^{i\bar{k},j\bar{l}}\partial_{\alpha}(\tilde{g}_{i\bar{k}})\partial_{\beta}(\tilde{g}_{j\bar{l}})
-2F^{i\bar{i}}\mathrm{Re}([\partial_{\beta}, e_{i}]\partial_{\alpha}\bar{e}_{i}u)
-2F^{i\bar{i}}\mathrm{Re}([\partial_{\alpha}, e_{i}] \partial_{\beta}\bar{e}_{i}u)+O(\lambda_{1})\mathcal{F}.
\end{split}
\end{equation*}
Using \eqref{second derive of F}, the concavity of $f$  and $\rho>0$, we have	
\[
-2F^{i\bar{j},p\bar{q}}\rho_{\alpha\beta}\partial_{\alpha}(\tilde{g}_{i\bar{j}})\partial_{\beta}(\tilde{g}_{p\bar{q}})\geq 0,
\]
and so
\begin{equation*}
\begin{split}
L(|\rho|^{2}) = {} &
2\sum_{\alpha,\beta}F^{i\bar{i}}|e_{i}(u_{\alpha\beta})|^{2}
+2\sum_{\alpha,\beta}F^{i\bar{i}}(e_{i}\bar e_{i}-[e_{i},\bar{e}_{i}]^{(0,1)})(u_{\alpha\beta})\rho_{\alpha\beta} \\
& +\sum_{\alpha,\beta,\gamma,\delta}F^{i\ov{i}}e_{i}e_{\ov{i}}(\chi^{\alpha\gamma}\chi^{\beta\delta})\rho_{\alpha\beta}\rho_{\gamma\delta} \\
\geq {} &
\sum_{\alpha,\beta}2F^{i\bar{i}}|e_{i}(u_{\alpha\beta})|^{2}-2\sum_{\alpha,\beta}\rho_{\alpha\beta}F^{i\bar{i}}\mathrm{Re}([\partial_{\beta}, e_{i}]\partial_{\alpha}\bar{e}_{i}u)
\\[0.5mm]
& -2\sum_{\alpha,\beta}\rho_{\alpha\beta}F^{i\bar{i}}\mathrm{Re}([\partial_{\alpha}, e_{i}] \partial_{\beta}\bar{e}_{i}u)-CN^{2}\mathcal{F}.
\end{split}
\end{equation*}
Using the similar argument of Claim 2 in Lemma \ref{lower bound of L lambda1},
\begin{equation*}
\begin{split}
& 2\sum_{\alpha,\beta}\rho_{\alpha\beta}F^{i\bar{i}}\mathrm{Re}([\partial_{\beta}, e_{i}]\partial_{\alpha}\bar{e}_{i}u)
+2\sum_{\alpha,\beta}\rho_{\alpha\beta}F^{i\bar{i}}\mathrm{Re}([\partial_{\alpha}, e_{i}] \partial_{\beta}\bar{e}_{i}u) \\
\leq {} & \ve\sum_{\alpha,\beta}F^{i\bar{i}}|e_{i}(u_{\alpha\beta})|^{2}+\frac{C}{\varepsilon}N^{2}\mathcal{F}.
\end{split}
\end{equation*}
Then we obtain Lemma \ref{lower bound of L rho}.
\end{proof}

\subsubsection{Lower bound of $L(|\de u|^{2})$}
\begin{lemma}\label{lower bound of L de u}
At $x_{0}$, we have
\begin{equation}\label{3.7}
L(|\partial u|^{2}) \geq \frac{3}{4} \sum_{i,j}F^{i\bar{i}}(|e_{i}e_{j}u|^{2}+|e_{i}\bar{e}_{j}u|^{2})-CK\mathcal{F}.
\end{equation}
\end{lemma}

\begin{proof}
Direct calculation shows
\[
L(|\partial u|^{2})=F^{i\bar{i}}\big({e_{i}e_{\bar{i}}}(|\partial u|^{2})-[e_{i},\bar{e}_{i}]^{(0,1)}(|\partial u|^{2})\big) = I_{1}+I_{2}+I_{3},
\]
where
\[
\begin{split}
I_{1}
:= {} & F^{i\bar{i}}(e_{i}\bar{e}_{i}e_{j}u
-[e_{i},\bar{e}_{i}]^{(0,1)}e_{j}u)\bar{e}_{j}u, \\
I_{2}
:= {} & F^{i\bar{i}}(e_{i}\bar{e}_{i}\bar{e}_{j}u
-[e_{i},\bar{e}_{i}]^{(0,1)}\bar{e}_{j}u)e_{j}u,\\
I_{3}
:= {} & F^{i\bar{i}}(|e_{i}e_{j}u|^{2}
+|e_{i}\bar{e}_{j}u|^{2}).
\end{split}
\]
Applying $e_{j}$ to the equation \eqref{nonlinear equation},
\[
F^{i\bar{i}}(e_{j}e_{i}\bar{e}_{i}u-e_{j}[e_{i},\bar{e}_{i}]^{(0,1)}u)=h_{j}.
\]
Note that
\[
\begin{split}
& F^{i\bar{i}}(e_{i}\bar{e}_{i}e_{j}u-[e_{i},\bar{e}_{i}]^{(0,1)}e_{j}u) \\
= {} &  F^{i\bar{i}}(e_{j}e_{i}\bar{e}_{i}u+e_{i}[\bar{e}_{i},e_{j}]u
+[e_{i},e_{j}]\bar{e}_{i}u-[e_{i},\bar{e}_{i}]^{(0,1)}e_{j}u)\\
= {} & h_{j}+F^{i\bar{i}}e_{j}[e_{i},\bar{e}_{i}]^{(0,1)}u
+F^{i\bar{i}}(e_{i}[\bar{e}_{i},e_{j}]u
+[e_{i},e_{j}]\bar{e}_{i}u-[e_{i},\bar{e}_{i}]^{(0,1)}e_{j}u)\\
= {} & h_{j}+F^{i\bar{i}}\big\{e_{i}[\bar{e}_{i},e_{j}]u
+\bar{e}_{i}[e_{i},e_{j}]u+[[e_{i},e_{j}],\bar{e}_{i}]u
-[[e_{i},\bar{e}_{i}]^{(0,1)},e_{j}]u\big\}.
\end{split}
\]
Similarly,
\[
\begin{split}
& F^{i\bar{i}}(e_{i}\bar{e}_{i}e_{j}u-[e_{i},\bar{e}_{i}]^{(0,1)}e_{j}u) \\
= {} & h_{\ov{j}}+F^{i\bar{i}}\big\{e_{i}[\bar{e}_{i},\ov{e}_{j}]u
+\bar{e}_{i}[e_{i},\ov{e}_{j}]u+[[e_{i},\ov{e}_{j}],\bar{e}_{i}]u
-[[e_{i},\bar{e}_{i}]^{(0,1)},\ov{e}_{j}]u\big\}.
\end{split}
\]
By the Cauchy-Schwarz inequality,
\begin{equation}\label{4.9}
\begin{split}
& I_{1}+I_{2} \\
\geq {} & 2\textrm{Re}\Bigg(\sum_{j}h_{j}u_{\bar{j}}\Bigg)
-C|\partial u|\sum_{i,j}F^{i\bar{i}}(|e_{i}e_{j}u|+|e_{i}\bar{e}_{j}u|)-C|\partial u|^{2} \mathcal{F}\\
\geq {} & -C|\de u|
-\frac{1}{4}\sum_{j}F^{i\bar{i}}(|e_{i}e_{j}u|^{2}+|e_{i}\bar{e}_{j}u|^{2})-C|\partial u|^{2}\mathcal{F}.
\end{split}
\end{equation}
Then
\[
\begin{split}
L(|\partial u|^{2}) = {} & I_{1}+I_{2}+I_{3} \\[3mm]
\geq {} & \frac{3}{4}\sum_{i,j}F^{i\bar{i}}(|e_{i}e_{j}u|^{2}+|e_{i}\bar{e}_{j}u|^{2})
-C(|\de u|+|\de u|^{2})\mathcal{F} \\
\geq {} & \frac{3}{4}\sum_{i,j}F^{i\bar{i}}(|e_{i}e_{j}u|^{2}+|e_{i}\bar{e}_{j}u|^{2})-CK\mathcal{F}.
\end{split}
\]
\end{proof}

\subsubsection{Proof of Proposition \ref{lower bound of L Q}}
We will use the above computations to prove Proposition \ref{lower bound of L Q}.

\begin{proof}[Proof of Proposition \ref{lower bound of L Q}]
Combining \eqref{L Q}, Lemma \ref{lower bound of L lambda1}, \ref{lower bound of L rho} and \ref{lower bound of L de u}, we obtain
\begin{equation*}
\begin{split}
0 \geq {} & (2-\ve)
\sum_{\alpha>1}\frac{F^{i\bar{i}}|e_{i}(u_{V_{\alpha}V_{1}})|^{2}}{\lambda_{1}(\lambda_{1}-\lambda_{\alpha})}
-\frac{1}{\lambda_{1}}F^{i\bar{k},j\bar{l}}V_{1}(\tilde{g}_{i\bar{k}})V_{1}(\tilde{g}_{j\bar{l}})\\
&+(2-\varepsilon)\xi'\sum_{\alpha,\beta}F^{i\bar{i}}|e_{i}(u_{\alpha\beta})|^{2}
-(1+\ve)\frac{F^{i\bar{i}}|(\lambda_{1})_{i}|^{2}}{\lambda_{1}^{2}} +\xi''F^{i\bar{i}}|e_{i}(|\rho|^{2})|^{2}\\
& +\frac{3\eta'}{4}\sum_{i,j} F^{i\bar{i}}(|e_{i}e_{j}u|^{2}+|e_{i}\bar{e}_{j}u|^{2})
+\eta'' F^{i\bar{i}}|e_{i}(|\partial u|^{2})|^{2}\\
& -Ae^{-Au}L(u)+A^{2}e^{-Au} F^{i\bar{i}}|u_{i}|^{2}
-\frac{C}{\ve}(1+\xi'N^{2}+\eta'K)\mathcal{F}.
\end{split}
\end{equation*}
The first, second and fourth term are $G_{1}$, $G_{2}$ and $B$ respectively. It suffices to deal with the third and last term. For the third term, using \eqref{xieta} and \eqref{N lambda 1},
\[
(2-\varepsilon)\xi'\sum_{\alpha,\beta}F^{i\bar{i}}|e_{i}(u_{\alpha\beta})|^{2}
\geq \sum_{\alpha,\beta}\frac{F^{i\ov{i}}|e_{i}(u_{\alpha\beta})|^{2}}{20N^{2}}
\geq \sum_{\alpha,\beta}\frac{F^{i\ov{i}}|e_{i}(u_{\alpha\beta})|^{2}}{C_{A}\lambda_{1}^{2}}.
\]
For the last term, using \eqref{xieta} again,
\[
-\frac{C}{\ve}(1+\xi'N^{2}+\eta'K)\mathcal{F}
\geq -\frac{C}{\ve}\mathcal{F}.
\]
Combining the above inequalities, we obtain Proposition \ref{lower bound of L Q}.
\end{proof}

\subsection{Proof of Theorem \ref{Thm4.1}}
To prove Theorem \ref{Thm4.1}, we define the index set:
\begin{equation*}
 J:=\Big\{ 1\leq k\leq n :  \frac{\eta'}{2}\sum_{j} (|e_{k}e_{j}u|^{2}+|e_{k}\bar{e}_{j}u|^{2})\geq A^{5n}e^{-5nu}K \ \,  \text{at $x_{0}$}\Big\}.
\end{equation*}
If $J=\emptyset$, then we obtain Theorem \ref{Thm4.1} directly. So we assume $J\neq\emptyset$ and let $j_0$ be the maximal element of $J$. If $j_{0}<n$, then we define another index set:
\begin{equation}\label{def of S}
S:=\Big\{j_{0}\leq i\leq n-1 :
F^{i\bar{i}} \leq A^{-2}e^{2Au}F^{i+1\overline{i+1}} \ \,  \text{at $x_{0}$}\Big\}.
\end{equation}
According to the index sets $J$ and $S$, the proof of Theorem \ref{Thm4.1} can be divided into three cases:

\medskip
{\bf Case 1.} $j_{0}=n$.

\medskip
{\bf Case 2.} $j_{0}<n$ and $S=\emptyset$.

\medskip
{\bf Case 3.} $j_{0}<n$ and $S\neq\emptyset$.
\medskip

Actually, Case 3 is the most difficult case. Case 1 and 2 are relatively easy, and their arguments are very similar.

\subsubsection{Proofs of Case 1 and 2}
In Case 1, we choose $\ve=\frac{1}{3}$. By (\ref{5..10}) and the elementary inequality
\begin{equation}\label{ele}
	|a+b+c|^{2}\leq 3|a|^{2}+3|b|^{2}+3|c|^{2},
\end{equation}
we get
\begin{equation*}
\begin{split}
B = {} & -(1+\ve)\frac{F^{i\bar{i}}|(\lambda_{1})_{i}|^{2}}{\lambda_{1}^{2}}\\
\geq {} & -4KA^{2}e^{-2Au}\mathcal{F}-4(\xi')^{2} F^{i\bar{i}}|e_{i}(|\rho |^{2})|^{2}-4(\eta')^{2} F^{i\bar{i}}|e_{i}(|\partial u|^{2})|^{2}.	
\end{split}
\end{equation*}
Substituting this into (\ref{4.7"}), dropping the non-negative terms $G_{i}$ ($i=1,2,3$) and $A^{2}e^{-Au} F^{i\bar{i}}|u_{i}|^{2}$, and using \eqref{xieta}, we obtain
\begin{equation}\label{(2)3.17}
0 \geq \frac{3\eta'}{4}\sum_{i,j} F^{i\bar{i}}(|e_{i}e_{j}u|^{2}+|e_{i}\bar{e}_{j}u|^{2})
-\left(C+4KA^{2}e^{-2Au}\right)\mathcal{F}-Ae^{-Au}L(u).
\end{equation}
The assumption of Case 1 shows $n=j_{0}\in J$ and so
\begin{equation}\label{case1}
 \begin{split}
& \frac{\eta'}{2}\sum_{i,j} F^{i\bar{i}}(|e_{i}e_{j}u|^{2}+|e_{i}\bar{e}_{j}u|^{2})
\geq \frac{\eta'}{2}F^{n\bar{n}}\sum_{j}(|e_{n}e_{j}u|^{2}+|e_{n}\bar{e}_{j}u|^{2}) \\
\geq {} & A^{5n}e^{-5nu}KF^{n\bar{n}}
\geq \frac{1}{n}A^{5n}e^{-5nu}K\mathcal{F},
\end{split}
\end{equation}
where we used \eqref{F ii} in the last inequality. It then follows that
\[
\begin{split}
0 \geq {} & \frac{\eta'}{4}\sum_{i,j} F^{i\bar{i}}(|e_{i}e_{j}u|^{2}+|e_{i}\bar{e}_{j}u|^{2})-\left(C
+4KA^{2}e^{-2Au}\right)\mathcal{F}\\
&+ \frac{1}{n} A^{5n}e^{-5nu}K\mathcal{F}-Ae^{-Au}L(u).
\end{split}
\]
For the last term, by the Cauchy-Schwarz inequality, we obtain
\begin{equation}\label{lower bound of Leta}
L(u) = \sum_{i}F^{i\bar{i}}(e_{i}\ov{e}_{i}u-[e_{i},\ov{e}_{i}]^{(0,1)}u)
\leq \frac{\eta'}{4}F^{i\bar i}|e_{i}\ov{e}_{i}u|^{2}+CK\mathcal{F}
\end{equation}
and so
\begin{equation}\label{z}
0\geq\frac{1}{n} A^{5n}e^{-5Anu}K\mathcal{F}-(4A^{2}e^{-2Au}+C)K\mathcal{F}.
\end{equation}
Recalling $\sup_{M}u=-1$ and increasing $A$ if necessary, this yields a contradiction.

\bigskip

In Case 2, using $S=\emptyset$, we see that
\begin{equation*}
F^{j_{0}\bar{j_{0}}}\geq A^{-2}e^{2Au(x_{0})}F^{j_{0}+1\overline{j_{0}+1}}\geq \cdots \geq A^{-2n}e^{2nAu(x_{0})}F^{n\bar{n}}.
\end{equation*}
Combining this with \eqref{F ii},
\begin{equation}\label{case2}
\begin{split}
&\frac{\eta'}{2}\sum_{i,j}F^{i\bar{i}} (|e_{i}e_{j}u|^{2}+|e_{i}\bar{e}_{j}u|^{2})
\geq \frac{\eta'}{2}F^{j_{0}\bar{j_{0}}}\sum_{j}(|e_{j_{0}}e_{j}u|^{2}+|e_{j_{0}}\bar{e}_{j}u|^{2})\\
&\geq A^{5n}e^{-5nu}KF^{j_{0}\bar{j_{0}}}\geq A^{3n}e^{-3nu}KF^{n\bar{n}}\geq \frac{1}{n} A^{3n}e^{-3nu}K\mathcal{F}.
	\end{split}
\end{equation}
Replacing \eqref{case1} by \eqref{case2}, and using the similar argument of Case 1, we complete the proof of Case 2.

\subsubsection{Proof of Case 3}
In Case 3, we have $S\neq\emptyset$. Let $i_{0}$ be the minimal element of $S$. Using $i_{0}$, we define another nonempty set
\begin{equation*}
I=\{i_{0}+1,\cdots, n\}.
\end{equation*}
Now we decompose the term $B$ into three terms based on $I$:
\begin{equation}\label{definition Bi}
\begin{split}
B & = (1+\ve)\sum_{i}\frac{F^{i\bar{i}}|(\lambda_{1})_{i}|^{2}}{\lambda_{1}^{2}}\\
& = (1+\ve)\sum_{i\not\in I}\frac{F^{i\bar{i}}|(\lambda_{1})_{i}|^{2}}{\lambda_{1}^{2}}
+3\ve\sum_{i\in I}\frac{F^{i\bar{i}}|(\lambda_{1})_{i}|^{2}}{\lambda_{1}^{2}}+(1-2\ve)\sum_{i\in I}\frac{F^{i\bar{i}}|(\lambda_{1})_{i}|^{2}}{\lambda_{1}^{2}}\\[2mm]
& =: B_{1}+B_{2}+B_{3}.
\end{split}
\end{equation}

\subsubsection*{$\bullet$ Terms $B_{1}$ and $B_{2}$}

We first deal with the terms $B_{1}$ and $B_{2}$.

\begin{lemma}\label{bad terms 1 2}
At $x_{0}$, we have
\begin{equation*}
\begin{split}
B_{1}+B_{2}
\leq {} & \frac{\eta'}{4}\sum_{i,j}F^{i\bar{i}} (|e_{i}e_{j}u|^{2}+|e_{i}\bar{e}_{j}u|^{2})+\xi''F^{i\bar{i}}|e_{i}(|\rho|^{2})|^{2}\\&+\eta''F^{i\bar{i}}|e_{i}(|\partial u|^{2})|^{2}+9\ve A^{2}e^{-2Au}F^{i\bar{i}}|u_{i}|^{2}.
	\end{split}
\end{equation*}
\end{lemma}

\begin{proof}
Using \eqref{5..10}, the elementary inequality \eqref{ele} and $\ve\in(0,\frac{1}{3}]$, we obtain
\begin{equation*}
\begin{split}
B_{1}&=(1+\ve)\sum_{i\not\in I}F^{i\bar{i}}\left|-\xi'e_{i}(|\rho|^{2})-\eta'e_{i}(|\partial u|^{2})+Ae^{-Au}u_{i}\right|^{2}\\
	&\leq 4(\xi')^{2}\sum_{i\not\in I}F^{i\bar{i}}|e_{i}(|\rho|^{2})|^{2}+ 4(\eta')^{2}\sum_{i\not\in I}F^{i\bar{i}}|e_{i}(|\partial u|^{2})|^{2}+4A^{2}e^{-2Au}\sum_{i\not\in I}F^{i\bar{i}}|u_{i}|^{2}\\
	&\leq \xi''\sum_{i\not\in I}F^{i\bar{i}}|e_{i}(|\rho|^{2})|^{2}+ \eta''\sum_{i\not\in I}F^{i\bar{i}}|e_{i}(|\partial u|^{2})|^{2}+4A^{2}e^{-2Au}K\sum_{i\not\in I}F^{i\bar{i}},
\end{split}
\end{equation*}
where we used \eqref{xieta} in the last line. Similarly,
\begin{equation*}
	\begin{split}
		B_{2}&=3\ve \sum_{i\in I}F^{i\bar{i}}\left|-\xi'e_{i}(|\rho|^{2})-\eta'e_{i}(|\partial u|^{2})+Ae^{-Au}u_{i}\right|^{2}\\
		&\leq 9\ve (\xi')^{2}\sum_{i\in I}F^{i\bar{i}}|e_{i}(|\rho|^{2})|^{2}+ 9\ve (\eta')^{2}\sum_{i\in I}F^{i\bar{i}}|e_{i}(|\de u|^{2})|^{2}+9\ve A^{2}e^{-2Au}\sum_{i\in I}F^{i\bar{i}}|u_{i}|^{2}\\
		&\leq \xi''\sum_{i\in I}F^{i\bar{i}}|e_{i}(|\rho|^{2})|^{2}+ \eta''\sum_{i\in I}F^{i\bar{i}}|e_{i}(|\de u|^{2})|^{2}+9\ve A^{2}e^{-2Au}F^{i\bar{i}}|u_{i}|^{2}.
	\end{split}
\end{equation*}
It then follows that
\begin{equation*}
\begin{split}
B_{1}+B_{2}
\leq {} & 4A^{2}e^{-2Au}K\sum_{i\not\in I}F^{i\bar{i}}+\xi''F^{i\bar{i}}|e_{i}(|\rho|^{2})|^{2}\\&+\eta''F^{i\bar{i}}|e_{i}(|\partial u|^{2})|^{2}
+9\ve A^{2}e^{-2Au}F^{i\bar{i}}|u_{i}|^{2}.
\end{split}
\end{equation*}
Since $I=\{i_{0}+1,\cdots, n\}$, then for each $i\not\in I$, we have $i\leq i_{0}$ and $F^{i\bar{i}}\leq F^{i_{0}\bar{i_{0}}}$ (see \eqref{F ii}). This shows
\begin{equation}\label{bad terms 1 2 eqn 1}
\begin{split}
B_{1}+B_{2}
\leq {} & 4nA^{2}e^{-2Au}KF^{i_{0}\bar{i_{0}}}
+\xi''F^{i\bar{i}}|e_{i}(|\rho|^{2})|^{2}\\
&+\eta''F^{i\bar{i}}|e_{i}(|\partial u|^{2})|^{2}+9\ve A^{2}e^{-2Au}F^{i\bar{i}}|u_{i}|^{2}.
\end{split}
\end{equation}

On the other hand, we assert that
\begin{equation}\label{bad terms 1 2 eqn 2}
F^{j_{0}\bar{j_{0}}} > A^{-2n}e^{2nAu} F^{i_{0}\bar{i_{0}}}.
\end{equation}
The definitions of $i_{0}$ and $j_{0}$ show $j_{0}\leq i_{0}$. If $j_{0}=i_{0}$, the above is trivial. If $j_{0}\leq i_{0}-1$, since $i_{0}$ is the minimal element of $S$, then we obtain $j_{0},\ldots,i_{0}-1\not\in S$ and
\begin{equation*}
F^{j_{0}\bar{j_{0}}}
> A^{-2}e^{2Au}F^{j_{0}+1\overline{j_{0}+1}}
> \cdots
> A^{-2(i_{0}-j_{0})}e^{2(i_{0}-j_{0})Au}F^{i_{0}\bar{i_{0}}}
\geq A^{-2n}e^{2nAu}F^{i_{0}\bar{i_{0}}}.
\end{equation*}
Using $j_{0}\in J$ and \eqref{bad terms 1 2 eqn 2},
\begin{equation}\label{bad terms 1 2 eqn 3}
\begin{split}
& \frac{\eta'}{4}\sum_{i,j}F^{i\bar{i}} (|e_{i}e_{j}u|^{2}+|e_{i}\bar{e}_{j}u|^{2})	
\geq \frac{\eta'}{4}F^{j_{0}\bar{j_{0}}}\sum_{j}(|e_{j_{0}}e_{j}u|^{2}+|e_{j_{0}}\bar{e}_{j}u|^{2})	\\
\geq {} & \frac{1}{2}A^{5n}e^{-5nAu}K F^{j_{0}\bar{j_{0}}}
\geq \frac{1}{2}A^{3n}e^{-3nAu}KF^{i_{0}\bar{i_{0}}}.
\end{split}	
\end{equation}
Combining \eqref{bad terms 1 2 eqn 1} and \eqref{bad terms 1 2 eqn 3}, and increasing $A$ if necessary, we are done.
\end{proof}

\subsubsection*{$\bullet$ Term $B_{3}$}
We will use $G_{1}$, $G_{2}$ and $G_{3}$ to control the term $B_{3}$.
\begin{lemma}\label{B 3}
If $ \ve=\frac{e^{Au(x_{0})}}{9}$, then at $x_{0}$, we have
\[
B_{3} \le G_{1}+G_{2}+G_{3}+\frac{C}{\ve}\mathcal{F}.
\]
\end{lemma}
We first define
\begin{equation}\label{definition of W}
W_{1} = \frac{1}{\sqrt{2}}(V_{1}-\sqrt{-1}JV_{1}), \quad
W_{1} = \sum_{q}\nu_{q}e_{q}, \quad
JV_{1} = \sum_{\alpha>1}\varsigma_{\alpha}V_{\alpha},
\end{equation}
where we used $V_1$ is orthogonal to $JV_1$. At $x_{0}$, since $V_{1}$ and $e_{q}$ are $\chi$-unit, then
\[
\sum_{q}|\nu_{q}|^{2} = 1, \quad
\sum_{\alpha>1}\varsigma_{\alpha}^{2} = 1.
\]

\begin{lemma}\label{nu}
At $x_{0}$, we have
\begin{enumerate}\setlength{\itemsep}{1mm}
	\item $\omega_{u}\geq -C_{A}K\chi$,
	\item $|\nu_{i}|\leq \frac{C_{A}K}{\lambda_{1}}$, for $i\in I$.
\end{enumerate}
\end{lemma}

\begin{proof}
For (1), recall that the maximal element of $J$ is $j_{0}$, and the minimal element of $I$ is $i_{0}+1$  satisfying $i_{0}+1>j_{0}$. This shows $I\cap J=\emptyset$, i.e. for any $i\in I$, we have $i\not\in J$. Thus,
\begin{equation}\label{nu eqn 1}
\frac{\eta'}{4}\sum_{j} (|e_{i}e_{j}u|^{2}+|e_{i}\bar{e}_{j}u|^{2}) \leq \frac{1}{2}A^{5n}e^{-5Anu}K,
\quad \text{for $i\in I$}.
\end{equation}
Since $n\in I$, then \eqref{nu eqn 1} implies $e_{n}\ov{e}_{n}u\geq-C_{A}K$ and so
\[
\ti{g}_{n\ov{n}} = g_{n\ov{n}}+e_{n}\ov{e}_{n}u-[e_{n},\ov{e}_{n}]^{(0,1)}u
\geq e_{n}\ov{e}_{n}u-CK \geq -C_{A}K.
\]
Combining this with \eqref{5..6}, we obtain $\omega_{u}\geq -C_{A}K\chi$.

For (2), use \eqref{nu eqn 1} again, we see that
\[
\sum_{\gamma=2i_{0}+1}^{2n}\sum_{\beta=1}^{2n}|u_{\gamma\beta}|\leq C_{A} K.
\]
It then follows that
\[
|\Phi_{\beta}^{\gamma}| = |u_{\beta\gamma}| \leq C_{A}K, \quad \text{for $2i_{0}+1\leq\gamma\leq 2n$, $1\leq \beta\leq 2n$}.
\]
Recalling $\Phi(V_{1})=\lambda_{1}V_{1}$, we obtain
\[
|V_{1}^{\gamma}|
= \Bigg|\frac{1}{\lambda_{1}}\sum_{\beta}\Phi_{\beta}^{\gamma}V_{1}^{\beta}\Bigg|
\leq \frac{C_{A}K}{\lambda_{1}}, \quad \text{for $2i_{0}+1\leq\gamma\leq 2n$},
\]
which implies
\[
|\nu_{i}|\leq |V_{1}^{2i-1}|+|V^{2i}_{1}|\leq \frac{C_{A}K}{\lambda_{1}}, \quad
\text{for $i\in I$}.
\]
\end{proof}

The definition of $W_{1}$ \eqref{definition of W} gives $V_{1}=-\sqrt{-1}JV_{1}+\sqrt{2}\,\ov{W_{1}}$. Then direct calculation shows
\[
\begin{split}
e_{i}(u_{V_1V_1})
= {} & -\sqrt{-1}e_{i}(u_{V_{1}JV_{1}})+\sqrt{2}e_{i}(u_{V_{1}\ov{W_{1}}}) \\[3mm]
= {} & -\sqrt{-1}\sum_{\alpha>1}\varsigma_{\alpha} e_{i}(u_{V_1 V_\alpha})
+\sqrt{2}\sum_{q}\ov{\nu_q}V_{1}e_{i}\ov{e}_{q}u+O(\lambda_1) \\
= {} & -\sqrt{-1}\sum_{\alpha>1}\varsigma_{\alpha} e_{i}(u_{V_1 V_\alpha})
+\sqrt{2}\sum_{q\notin I}\ov{\nu_q}V_1(\tilde{g}_{i\ov{q}})
+\sqrt{2}\sum_{q\in I}\ov{\nu_q}V_{1}e_{i}\ov{e}_{q}u+O(\lambda_1).
\end{split}
\]
Combining this with the Cauchy-Schwarz inequality and Lemma \ref{nu},
\[
\begin{split}
B_{3} = {} & (1-2\ve)\sum_{i\in I}\frac{F^{i\ov{i}}|e_{i}(u_{V_{1}V_{1}})|^2}{\lambda_1^2} \\
\leq {} & (1-\ve)\sum_{i\in I}\frac{F^{i\ov{i}}}{\lambda_1^2}
\left|-\sqrt{-1}\sum_{\alpha>1}\varsigma_{\alpha} e_{i}(u_{V_1 V_\alpha})+\sqrt{2}\sum_{q\notin I}\ov{\nu_q}V_{1}(\tilde{g}_{i\ov{q}})\right|^{2} \\
& + \frac{C_{A}K^{2}}{\ve\lambda_1^2}\sum_{i\in I}\sum_{q\in I}\frac{F^{i\ov{i}}|V_{1}e_{i}\ov{e}_{q}u|^{2}}{\lambda_1^2}+\frac{C\mathcal{F}}{\ve}.
\end{split}
\]
For any $\gamma>0$, using the Cauchy-Schwarz inequality again,
\[
\begin{split}
B_{3} \leq {} & (1-\ve)\left(1+\frac{1}{\gamma}\right)\sum_{i\in I}\frac{F^{i\ov{i}}}{\lambda_1^2}
\left|\sum_{\alpha>1}\varsigma_{\alpha} e_{i}(u_{V_1 V_\alpha})\right|^{2} \\
& +(1-\ve)(1+\gamma)\sum_{i\in I}\frac{2F^{i\ov{i}}}{\lambda_{1}^{2}}\left|\sum_{q\notin I}\ov{\nu_q}V_1(\tilde{g}_{i\ov{q}})\right|^{2} \\
& + \frac{C_A K^{2}}{\ve\lambda_1^2}\sum_{i\in I}\sum_{q\in I}\frac{F^{i\ov{i}}|V_{1}e_{i}\ov{e}_{q}u|^{2}}{\lambda_1^2}+\frac{C\mathcal{F}}{\ve} \\
=: {} & B_{31}+B_{32}+B_{33}+\frac{C\mathcal{F}}{\ve}.
\end{split}
\]

Now we give the proof of Lemma \ref{B 3}.
\begin{proof}[Proof of Lemma \ref{B 3}]
We first deal with $B_{33}$. It is clear that
\[
|V_{1}e_{i}\ov{e}_{q}u| \leq C\sum_{\alpha,\beta}|e_{i}(u_{\alpha\beta})|+C\lambda_{1}
\]
and so
\[
B_{33} = \frac{C_A K^{2}}{\ve\lambda_1^2}\sum_{i\in I}\sum_{q\in I}\frac{F^{i\ov{i}}|V_{1}e_{i}\ov{e}_{q}u|^{2}}{\lambda_1^2}
\leq \frac{C_A K^{2}}{\ve\lambda_1^2}\sum_{\alpha,\beta}\frac{F^{i\ov{i}}|e_{i}(u_{\alpha\beta})|^{2}}{\lambda_{1}^{2}}
+\frac{C_A K^{2}}{\ve\lambda_1^2}\mathcal{F}.
\]
Without loss of generality, we assume that $\lambda_{1}\geq\frac{C_{A}K}{\ve}$, which implies
\[
B_{33} \leq
\sum_{\alpha,\beta}\frac{F^{i\ov{i}}|e_{i}(u_{\alpha\beta})|^{2}}{C_{A}\lambda_{1}^{2}}+\mathcal{F}
= G_{3}+\mathcal{F}.
\]
To prove Lemma \ref{B 3}, it suffices to show
\begin{equation}\label{B 3 claim}
B_{31}+B_{32} \leq G_{1}+G_{2}.
\end{equation}
For the term $B_{31}$, by the Cauchy-Schwarz inequality,
\begin{equation}\label{B31}
\begin{split}
B_{31} \leq {} & (1-\ve)\left(1+\frac{1}{\gamma}\right)
\sum_{i\in I}\frac{F^{i\ov{i}}}{\lambda_1^2}\left(\sum_{\alpha>1}(\lambda_{1}-\lambda_{\alpha})\varsigma_{\alpha}^{2}\right)
\left(\sum_{\alpha>1}\frac{|e_{i}(u_{V_{1}V_{\alpha}})|^{2}}{\lambda_1-\lambda_\alpha}\right) \\
= {} & (1-\ve)\left(1+\frac{1}{\gamma}\right)\sum_{i\in I}\frac{F^{i\ov{i}}}{\lambda_1^2}\left(\lambda_{1}-\sum_{\alpha>1}\lambda_{\alpha}\varsigma_{\alpha}^{2}\right)
\left(\sum_{\alpha>1}\frac{|e_{i}(u_{V_{1}V_{\alpha}})|^{2}}{\lambda_1-\lambda_\alpha}\right) \\
\leq {} & \frac{1-\ve}{(2-\ve)\lambda_{1}}
\left(1+\frac{1}{\gamma}\right)\left(\lambda_{1}-\sum_{\alpha>1}\lambda_{\alpha}\varsigma_{\alpha}^{2}\right)G_{1}.
\end{split}
\end{equation}
For the term $B_{32}$, using the Cauchy-Schwarz inequality again,
\[
\begin{split}
    B_{32}= {} &
	(1-\ve)(1+\gamma)\sum_{i\in I}\frac{2F^{i\ov{i}}}{\lambda_{1}^{2}}\left|\sum_{q\notin I}\ov{\nu_q}V_1(\tilde{g}_{i\ov{q}})\right|^{2} \\
	\leq {} & (1-\ve)(1+\gamma)\sum_{i\in I}\frac{2F^{i\ov{i}}}{\lambda_{1}^{2}}
	\left(\sum_{q\notin I}
	\frac{(\tilde{g}_{q\ov{q}}-\tilde{g}_{i\ov{i}})|\nu_{q}|^{2}}{F^{i\ov{i}}-F^{q\ov{q}}}\right)
	\left(\sum_{q\notin I}
	\frac{(F^{i\ov{i}}-F^{q\ov{q}})|V_{1}(\tilde{g}_{i\ov{q}})|^{2}}{\tilde{g}_{q\ov{q}}-\tilde{g}_{i\ov{i}}}\right).
\end{split}
\]
 The denominators in the above are not zero. Indeed, for $i\in I$ and $q\notin I$, by definition of the index set $I$,
\[
F^{q\ov{q}} \leq F^{i_0\ov{i_0}} \leq A^{-2}e^{2Au}F^{i_0+1\ov{i_0+1}} \leq A^{-2}e^{2Au}F^{i\ov{i}}.
\]
Increasing $A$ if necessary, we have $F^{i\ov{i}}\neq F^{q\ov{q}}$ and so $\tilde{g}_{q\ov{q}}\neq\tilde{g}_{i\ov{i}}$. Recalling $\omega_{u}\geq -C_{A}K\chi$ (see Lemma \ref{nu}), for $i\in I$ and $q\notin I$, we have
\begin{equation}\label{positive constant}
0 \leq
\frac{(\tilde{g}_{q\ov{q}}-\tilde{g}_{i\ov{i}})|\nu_{q}|^{2}}{F^{i\ov{i}}-F^{q\ov{q}}}
\leq \frac{\tilde{g}_{q\ov{q}}|\nu_{q}|^{2}-\tilde{g}_{i\ov{i}}|\nu_{q}|^{2}}{(1-A^{-2}e^{2Au})F^{i\ov{i}}}
\leq \frac{\tilde{g}_{q\ov{q}}|\nu_{q}|^{2}+C_{A}K}{(1-A^{-2}e^{2Au})F^{i\ov{i}}}.
\end{equation}
It then follows that
\[
\begin{split}
B_{32}
\leq {} & (1-\ve)(1+\gamma)\sum_{i\in I}\frac{2F^{i\ov{i}}}{\lambda_{1}^{2}}
\left(\sum_{q\notin I}
\frac{\tilde{g}_{q\ov{q}}|\nu_{q}|^{2}+C_{A}K}{(1-A^{-2}e^{2Au})F^{i\ov{i}}}\right)
\left(\sum_{q\notin I} \frac{(F^{i\ov{i}}-F^{q\ov{q}})|V_{1}(\tilde{g}_{i\ov{q}})|^{2}}{\tilde{g}_{q\ov{q}}-\tilde{g}_{i\ov{i}}}\right) \\
\leq {} & \frac{(1-\ve)(1+\gamma)}{1-A^{-2}e^{2Au}}
\left(\sum_{q\notin I}\tilde{g}_{q\ov{q}}|\nu_{q}|^{2}+C_{A}K\right)
\sum_{i\in I}\frac{2}{\lambda_{1}^{2}}\left(\sum_{q\notin I}	\frac{(F^{i\ov{i}}-F^{q\ov{q}})|V_{1}(\tilde{g}_{i\ov{q}})|^{2}}{\tilde{g}_{q\ov{q}}-\tilde{g}_{i\ov{i}}}\right).
\end{split}
\]
Using \eqref{second derive of F} and the concavity of $f$, we have
\[
G_{2} = -\frac{1}{\lambda_{1}} F^{i\bar{k},j\bar{l}}V_{1}(\tilde{g}_{i\bar{k}})V_{1}(\tilde{g}_{j\bar{l}})
\geq \frac{2}{\lambda_{1}}\sum_{i\in I}\sum_{q\notin I}\frac{(F^{i\ov{i}}-F^{q\ov{q}})|V_{1}(\tilde{g}_{i\ov{q}})|^{2}}{\tilde{g}_{q\ov{q}}-\tilde{g}_{i\ov{i}}}
\]
and so
\[
\begin{split}
B_{32} \leq {} & \frac{(1-\ve)(1+\gamma )}{\lambda_{1}(1-A^{-2}e^{2Au})}
\left(\sum_{q\notin I}\tilde{g}_{q\ov{q}}|\nu_{q}|^{2}+C_{A}K\right)G_{2}.
\end{split}
\]
Increasing $A$ if necessary, $\ve=\frac{e^{Au(x_{0})}}{9}$ implies
\[
\frac{(1-\ve)(1+\gamma)}{\lambda_{1}(1-A^{-2}e^{2Au})}
\leq \left(1-\frac{\ve}{2}\right)\left(\frac{1+\gamma}{\lambda_{1}}\right).
\]
Then
\begin{equation}\label{B32}
B_{32} \leq \left(1-\frac{\ve}{2}\right)\left(\frac{1+\gamma}{\lambda_{1}}\right)
\left(\sum_{q\notin I}\tilde{g}_{q\ov{q}}|\nu_{q}|^{2}+C_{A}K\right)G_{2}.
\end{equation}
Combining \eqref{B31} and \eqref{B32},
\[
\begin{split}
B_{31}+B_{32}
\leq {} & \frac{1-\ve}{(2-\ve)\lambda_{1}}
\left(1+\frac{1}{\gamma}\right)\left(\lambda_{1}-\sum_{\alpha>1}\lambda_{\alpha}\varsigma_{\alpha}^{2}\right)G_{1} \\
& {} + \left(1-\frac{\ve}{2}\right)\left(\frac{1+\gamma}{\lambda_{1}}\right)
\left(\sum_{q\notin I}\tilde{g}_{q\ov{q}}|\nu_{q}|^{2}+C_{A}K\right)G_{2}.
\end{split}
\]
 Thanks to Lemma \ref{nu}, increasing $C_{A}$ if necessary, we may assume that
\begin{equation}\label{B 3 eqn 1}
\sum_{q\notin I}\tilde{g}_{q\ov{q}}|\nu_{q}|^{2}+C_{A}K > 0.
\end{equation}
We split the proof of \eqref{B 3 claim} into two cases.

\bigskip
\noindent
{\bf Case (i).} $\frac{1}{2}\left(\lambda_{1}+\sum_{\alpha>1}\lambda_{\alpha}\varsigma_{\alpha}^{2}\right)
>\left(1-\frac{\ve}{2}\right)\left(\sum_{q\notin I}\tilde{g}_{q\ov{q}}|\nu_{q}|^{2}+C_{A}K\right)$.
\bigskip

By the assumption of Case (i) and \eqref{B 3 eqn 1},
\[
\frac{1}{2}\left(\lambda_{1}+\sum_{\alpha>1}\lambda_{\alpha}\varsigma_{\alpha}^{2}\right)
> \left(1-\frac{\ve}{2}\right)\left(\sum_{q\notin I}\tilde{g}_{q\ov{q}}|\nu_{q}|^{2}+C_{A}K\right)
> 0.
\]
Since $\lambda_{1}>\lambda_{2}$ at $x_{0}$, we can choose
\[
\gamma := \frac{\lambda_{1}-\sum_{\alpha>1}\lambda_{\alpha}\varsigma_{\alpha}^{2}}
{\lambda_{1}+\sum_{\alpha>1}\lambda_{\alpha}\varsigma_{\alpha}^{2}} > 0.
\]
Then
\[
\begin{split}
B_{31}+B_{32}
\leq {} & \frac{1-\ve}{(2-\ve)\lambda_{1}}
\left(1+\frac{1}{\gamma}\right)\left(\lambda_{1}-\sum_{\alpha>1}\lambda_{\alpha}\varsigma_{\alpha}^{2}\right)
G_{1} \\
& {} + \left(1-\frac{\ve}{2}\right)\left(\frac{1+\gamma}{\lambda_{1}}\right)
\left(\sum_{q\notin I}\tilde{g}_{q\ov{q}}|\nu_{q}|^{2}+C_{A}K\right)G_{2} \\
\leq {} & \frac{1}{2\lambda_{1}}\left(1+\frac{1}{\gamma}\right)
\left(\lambda_{1}-\sum_{\alpha>1}\lambda_{\alpha}\varsigma_{\alpha}^{2}\right)
G_{1}+\frac{1+\gamma}{2\lambda_{1}}
\left(\lambda_{1}+\sum_{\alpha>1}\lambda_{\alpha}\varsigma_{\alpha}^{2}\right)G_{2} \\[2mm]
= {} & G_{1}+G_{2}.
\end{split}
\]

\bigskip
\noindent
{\bf Case (ii).} $\frac{1}{2}\left(\lambda_{1}+\sum_{\alpha>1}\lambda_{\alpha}\varsigma_{\alpha}^{2}\right)
\leq(1-\frac{\ve}{2})\left(\sum_{q\notin I}\tilde{g}_{q\ov{q}}|\nu_{q}|^{2}+C_{A}K\right)$.
\bigskip

It is clear that
\begin{equation}\label{case b inequality 2}
\begin{split}
& \tilde{g}(W_1, \ov{W_{1}}) = g(W_1, \ov{W_{1}})+(\de\dbar u)(W_1, \ov{W_{1}})
=  W_{1}\ov{W_{1}}(u)+O(K) \\
= {} & \frac{1}{2}\left(u_{V_{1}V_{1}}+u_{JV_{1}JV_{1}}\right)+O(K)
= \frac{1}{2}\left(\lambda_1+\sum_{\alpha>1}\lambda_\alpha\varsigma_{\alpha}^2\right)+O(K),
\end{split}
\end{equation}
where $O(K)$ denotes a term satisfying $|O(K)|\leq CK$ for some uniform constant $C$. Thanks to Lemma \ref{nu}, we have $\omega_{u}\geq-C_{A}K\chi$,
and then
\begin{equation*}
\begin{split}
&\sum_{q\notin I}\tilde{g}_{q\ov{q}}|\nu_{q}|^{2}+C_{A}K
\leq \sum_{q}\tilde{g}_{q\ov{q}}|\nu_{q}|^{2}+C_{A}K \\
= {} & \tilde{g}(W_1, \ov{W_{1}})+C_{A}K
\leq \frac{1}{2}\left(\lambda_1+\sum_{\alpha>1}\lambda_\alpha\varsigma_{\alpha}^2\right)+C_{A}K,
\end{split}
\end{equation*}
where we used \eqref{definition of W} in the last inequality. Combining this with the assumption of Case (ii), we see that
\begin{equation}\label{case b inequality 1}
\sum_{q\notin I}\tilde{g}_{q\ov{q}}|\nu_{q}|^{2}+C_{A}K \leq \frac{C_{A}K}{\ve}.
\end{equation}
Using $\omega_{u}\geq-C_{A}K\chi$ and \eqref{case b inequality 2},
\[
\frac{1}{2}\left(\lambda_1+\sum_{\alpha>1}\lambda_\alpha\varsigma_{\alpha}^2\right)
\geq \tilde{g}(W_1, \ov{W_{1}})-CK = \sum_{q}\tilde{g}_{q\ov{q}}|\nu_{q}|^{2}-CK \geq -C_{A}K.
\]
It then follows that
\[
0 < \lambda_{1}-\sum_{\alpha>1}\lambda_{\alpha}\varsigma_{\alpha}^{2}
\leq 2\lambda_{1}+C_{A}K\leq (2+2\ve^{2})\lambda_{1},
\]
where we assume without loss of generality that $\lambda_{1}\geq\frac{C_{A}K}{\ve^{2}}$. Choosing $\gamma=\frac{1}{\ve^{2}}$ and using \eqref{case b inequality 1},
\[
\begin{split}
B_{31}+B_{32}
\leq {} & \frac{1-\ve}{(2-\ve)\lambda_{1}}
\left(1+\frac{1}{\gamma}\right)\left(\lambda_{1}-\sum_{\alpha>1}\lambda_{\alpha}\varsigma_{\alpha}^{2}\right)G_{1} \\
& {} + \left(1-\frac{\ve}{2}\right)\left(\frac{1+\gamma}{\lambda_{1}}\right)
\left(\sum_{q\notin I}\tilde{g}_{q\ov{q}}|\nu_{q}|^{2}+C_{A}K\right)G_{2} \\
\leq {} & \frac{2-2\ve}{2-\ve}(1+\ve^{2})(1+\ve^{2})G_{1}
+\frac{C_{A}K}{\ve^{3}\lambda_{1}}G_{2}.
\end{split}
\]
Recalling $\ve=\frac{e^{Au(x_{0})}}{9}$ and increasing $A$ if necessary, we may assume that
\[
\frac{2-2\ve}{2-\ve}(1+\ve^{2})(1+\ve^{2}) \leq 1.
\]
Then $B_{31}+B_{32} \leq G_{1}+G_{2}$ provided that $\lambda_{1}\geq\frac{C_{A}}{\ve^{3}}$.
\end{proof}

Now we are in a position to prove Case 3 of Theorem \ref{Thm4.1}.

\begin{proof}[Proof of Case 3]
Combining Proposition \ref{lower bound of L Q}, Lemma \ref{bad terms 1 2} and \ref{B 3}, we have
\begin{equation*}
\begin{split}
0 \geq {} & (A^{2}e^{-Au}-9\ve A^{2}e^{-2Au}) F^{i\bar{i}}|u_{i}|^{2}-\frac{C}{\ve}\mathcal{F}\\
&+ \frac{\eta'}{4}\sum_{i,j} F^{i\bar{i}}(|e_{i}e_{j}u|^{2}+|e_{i}\bar{e}_{j}u|^{2})
-Ae^{-Au}L(u).
\end{split}
\end{equation*}
Recalling that $\ve$ has been chosen as $\frac{e^{Au(x_{0})}}{9}$ in Lemma \ref{B 3}, we obtain
\begin{equation}\label{proof of Thm4.1 eqn 1}
\begin{split}
0 \geq &-\frac{C}{\ve}\mathcal{F}+ \frac{\eta'}{4}\sum_{i,j} F^{i\bar{i}}(|e_{i}e_{j}u|^{2}+|e_{i}\bar{e}_{j}u|^{2})-Ae^{-Au}L(u).
\end{split}
\end{equation}
We now choose $A=\frac{9C+1}{\theta}$, where $\theta$ is the constant in Proposition \ref{prop subsolution}. There are two subcases:

\bigskip
\noindent
{\bf Subcase 3.1.} $-L(u)\geq \theta\mathcal{F}$.
\bigskip

In this subcase, \eqref{proof of Thm4.1 eqn 1} implies
\[
0 \geq \left(A\theta e^{-Au}-\frac{C}{\ve}\right)\mathcal{F}+ \frac{\eta'}{4}\sum_{i,j} F^{i\bar{i}}(|e_{i}e_{j}u|^{2}+|e_{i}\bar{e}_{j}u|^{2}).
\]
Using $A=\frac{9C+1}{\theta}$, we have
\[
A\theta e^{-Au}-\frac{C}{\ve}
= A\theta e^{-Au}-9Ce^{-Au} \geq e^{-Au}.
\]
It then follows that
\[
0\geq e^{-Au}\mathcal{F}+ \frac{\eta'}{4}\sum_{i,j} F^{i\bar{i}}(|e_{i}e_{j}u|^{2}+|e_{i}\bar{e}_{j}u|^{2}).
\]
This yields a contradiction.

\bigskip
\noindent
{\bf Subcase 3.2.} $F^{i\bar{i}}\geq \theta\mathcal{F}$ for $i=1,2,\ldots,n$,
\bigskip

By the Cauchy-Schwarz inequality,
\begin{equation*}
\begin{split}
Ae^{-Au}L(u) = {} & Ae^{-Au}\sum_{i}F^{i\bar{i}}(e_{i}\ov{e}_{i}u-[e_{i},\ov{e}_{i}]^{(0,1)}u) \\
\leq {} & Ae^{-Au}\mathcal{F}\sum_{i}|e_{i}\ov{e}_{i}u|+CAe^{-Au}K\mathcal{F} \\
\leq {} & \frac{\theta\eta'}{8}\mathcal{F}\sum_{i}|e_{i}\ov{e}_{i}u|^{2}+C_{A}K\mathcal{F}.
\end{split}
\end{equation*}
Substituting this into \eqref{proof of Thm4.1 eqn 1},
\[
\frac{\theta\eta'}{8}\mathcal{F}\sum_{i,j} (|e_{i}e_{j}u|^{2}+|e_{i}\bar{e}_{j}u|^{2})\leq C_{A} K\mathcal{F}.
\]
It then follows that
\[
\sum_{i,j} (|e_{i}e_{j}u|^{2}+|e_{i}\bar{e}_{j}u|^{2})\leq C_{A} K^{2}
\]
and so $\lambda_{1}\leq C_{A}K$. This completes the proof of Subcase 3.2.
\end{proof}

\section{Proof of Theorem \ref{main estimate}}\label{proof of main estimate}
In this section, we will prove the $C^{2}$ estimate by adapting the argument of \cite[Section 6]{Szekelyhidi18} in the almost Hermitian setting, and then give the proof of Theorem \ref{main estimate}.

\subsection{$C^{2}$ estimate}
\begin{proposition}\label{blowup argument}
Let $u$ be a smooth solution of the equation \eqref{nonlinear equation} with $\sup_{M}(u-\underline{u})=0$. Then there exists a constant $C$ depending on $\underline{u}$, $h$, $\omega$, $f$, $\Gamma$ and $(M,\chi,J)$ such that
\[
\sup_{M}|\nabla^{2}u|_{\chi} \leq C.
\]
\end{proposition}

\begin{proof}
By Theorem \ref{Thm4.1}, it suffices to prove
\begin{equation}\label{blowup argument eqn 4}
\sup_{M}|\de u|_{\chi}^{2} \leq C.
\end{equation}
We follow the argument of \cite[Section 6]{Szekelyhidi18}. For convenience, we denote $|\cdot|_{\chi}$ by $|\cdot|$. To prove \eqref{blowup argument eqn 4}, we argue by contradiction. Suppose that there exist sequences of smooth functions $u_{k}$ and $h_{k}$ such that
\begin{enumerate}\setlength{\itemsep}{1mm}
\item $\sup_{\de\Gamma}f<h_{k}\leq\sup_{\Gamma}f-C_{0}^{-1}$ and $\|h_{k}\|_{C^{2}}\leq C_{0}$,
\item $F(\omega_{u_{k}})=h_{k}$ and it admits a $\mathcal{C}$-subsolution $\underline{u}_{k}$ with $\sup_{M}(u_{k}-\underline{u}_{k})=0$ and $\|\underline{u}_{k}\|_{C^{4}}\leq C_{0}$,
\item $\sup_{M}|\de u_{k}|^{2}=:\Lambda_{k}^{2}\rightarrow\infty$,
\end{enumerate}
where $C_{0}$ is a positive constant independent of $k$. Let $p_{k}$ be the maximum point of $|\de u_{k}|^{2}$, i.e.
\[
|\de u_{k}|^{2}(p_{k}) = \sup_{M}|\de u_{k}|^{2}=\Lambda_{k}^{2}\rightarrow\infty.
\]
Passing to a subsequence, we may assume that $p_{k}$ converges to $p_{\infty}$. Near $p_{\infty}$, we choose the coordinate system $(U,\{x^{\alpha}\}_{\alpha=1}^{2n})$ such that at $p_{\infty}$,
\[
\chi_{\alpha\beta}
= \chi\left(\frac{\de}{\de x^{\alpha}},\frac{\de}{\de x^{\beta}}\right)
= \delta_{\alpha\beta}
\]
and
\[
J\left(\frac{\de}{\de x^{2i-1}}\right) = \frac{\de}{\de x^{2i}}, \quad
\text{for $i=1,2,\ldots,n$}.
\]
For convenience, we assume that $U$ contains the Euclidean ball of radius $2$ and write $\de_{\alpha}=\frac{\de}{\de x^{\alpha}}$ for $\alpha=1,2,\ldots,2n$. Let $\{e_{i}\}_{i=1}^{n}$ be a local $\chi$-unitary frame of $(1,0)$-vectors with respect to $J$ such that at $p_{\infty}$,
\[
e_{i} = \frac{1}{\sqrt{2}}
\left(\de_{2i-1}-\sqrt{-1}\de_{2i}\right), \quad \text{for $i=1,2,\cdots,n$}.
\]
Let $Z_{i}$ be the vector field defined by
\[
Z_{i} = \frac{1}{\sqrt{2}}
\left(\de_{2i-1}-\sqrt{-1}\de_{2i}\right), \quad \text{for $i=1,2,\cdots,n$}.
\]
Note that $e_{i}(p_{\infty})=Z_{i}(p_{\infty})$, but they may be different outside $p_{\infty}$. Let $g_{E}$ and $J_{E}$ be the standard metric and complex structure of the Euclidean space. Then $\{Z_{i}\}_{i=1}^{n}$ is the standard $g_{E}$-unitary frame of $(1,0)$-vectors with respect to $J_{E}$. Note that $Z_{i}$ may be not $(1,0)$-vector field with respect to $J$.
To characterize the difference of $e_{i}$ and $Z_{i}$, we write
\[
Z_{i} = e_{i}+A_{i}^{\alpha}\de_{\alpha}.
\]
Since $e_{i}(p_{\infty})=Z_{i}(p_{\infty})$, then for any $i$ and $\alpha$, we have the following limit of the coefficient $A_{i}^{\alpha}$:
\[
\lim_{p\rightarrow p_{\infty}}A_{i}^{\alpha}(p) = 0.
\]

Next, we compute
\[
\begin{split}
& Z_{i}\ov{Z}_{j}u_{k} = \left(e_{i}+A_{i}^{\alpha}\de_{\alpha}\right)
\left(\ov{e}_{j}+\ov{A}_{j}^{\beta}\de_{\beta}\right)u_{k} \\
= {} & e_{i}\ov{e}_{j}u_{k}+e_{i}(\ov{A}_{j}^{\beta}\de_{\beta})u_{k}
+A_{i}^{\alpha}\de_{\alpha}\ov{e}_{j}u_{k}
+(A_{i}^{\alpha}\de_{\alpha})(\ov{A}_{j}^{\beta}\de_{\beta})u_{k} \\[1mm]
= {} & e_{i}\ov{e}_{j}u_{k}
+(e_{i}\ov{A}_{j}^{\beta})(\de_{\beta}u_{k})+\ov{A}_{j}^{\beta}e_{i}\de_{\beta}u_{k}
+A_{i}^{\alpha}\de_{\alpha}\ov{e}_{j}u_{k}
+A_{i}^{\alpha}(\de_{\alpha}\ov{A}_{j}^{\beta})(\de_{\beta}u_{k})
+A_{i}^{\alpha}\ov{A}_{j}^{\beta}\de_{\alpha}\de_{\beta}u_{k}.
\end{split}
\]
Thanks to Theorem \ref{Thm4.1}, we have
\begin{equation}\label{blowup argument eqn 5}
\sup_{M}|\nabla^{2}u_{k}|
\leq C\sup_{M}|\nabla u_{k}|^{2}+C \leq C\Lambda_{k}^{2}
\end{equation}
and so
\begin{equation}\label{blowup argument eqn 1}
\begin{split}
Z_{i}\ov{Z}_{j}u_{k}
= {} & e_{i}\ov{e}_{j}u_{k}+O(\Lambda_{k})+o_{E}\cdot\Lambda_{k}^{2} \\
= {} & \omega_{u_{k}}(e_{i},\ov{e}_{j})-\omega(e_{i},\ov{e}_{j})
+[e_{i},\ov{e}_{j}]^{(0,1)}u_{k}+O(\Lambda_{k})+o_{E}\cdot\Lambda_{k}^{2} \\[0.5mm]
= {} & \omega_{u_{k}}(e_{i},\ov{e}_{j})+O(\Lambda_{k})+o_{E}\cdot\Lambda_{k}^{2},
\end{split}
\end{equation}
where $O(\Lambda_{k})$ denotes a term satisfying $|O(\Lambda_{k})|\leq C\Lambda_{k}$, and $o_{E}$ denotes a term satisfying $\lim_{p\rightarrow p_{\infty}}o_{E}=0$.

Define the function $v_{k}:B_{\Lambda_{k}}(0)\rightarrow\mathbb{R}$ by
\[
v_{k}(z) := u_{k}\left(p_{k}+\frac{z}{\Lambda_{k}}\right).
\]
By Proposition \ref{Prop3.2} and \eqref{blowup argument eqn 5}, we obtain
\begin{enumerate}\setlength{\itemsep}{1mm}
\item $\|v_{k}\|_{C^{2}(B_{\Lambda_{k}}(0))}\leq C$,
\item $|\de v_{k}|(0)=1+o(1)$, where $o(1)$ denotes a term satisfying $\lim_{k\rightarrow\infty}o(1)=0$.
\end{enumerate}
Passing to a subsequence, we may assume that $v_{k}$ converges to $v_{\infty}$ in $C_{\mathrm{loc}}^{1,1}(\mathbb{C}^{n})$, and so
\[
\|v_{\infty}\|_{C^{1,1}(\mathbb{C}^{n})}\leq C, \quad |\de v_{\infty}|(0)=1.
\]
We have the following claim:
\begin{claim*}
The function $v_{\infty}$ is a $\Gamma$-solution, i.e.,
\begin{enumerate}\setlength{\itemsep}{1mm}
\item If $\vp$ be a $C^{2}$ function such that $\vp\geq v_{\infty}$ near $z_{0}$ and $\vp(z_{0})=v_{\infty}(z_{0})$, then we have $\lambda\left[\frac{\de^{2}\vp}{\de z^{i}\de\ov{z}^{j}}(z_{0})\right]\in\ov{\Gamma}$,
\item If $\vp$ be a $C^{2}$ function such that $\vp\leq v_{\infty}$ near $z_{0}$ and $\vp(z_{0})=v_{\infty}(z_{0})$, then we have $\lambda\left[\frac{\de^{2}\vp}{\de z^{i}\de\ov{z}^{j}}(z_{0})\right]\in\mathbb{R}^{n}\setminus\Gamma$.
\end{enumerate}
\end{claim*}

\begin{proof}[Proof of Claim]
(1) For any $\ve>0$, we choose sufficiently large $k$ such that
\[
\|v_{k}-v_{\infty}\|_{C^{0}(B_{1}(z_{0}))} \leq \ve^{2}.
\]
Then there exist $a\in\mathbb{R}$ with $|a|\leq\ve$ and $z_{1}\in B_{\ve}(z_{0})$ satisfying
\[
\vp+\ve|z-z_{0}|^{2}+a \geq v_{k} \ \ \text{on $B_{\ve}(z_{0})$}
\]
and
\[
\vp(z_{1})+\ve|z_{1}-z_{0}|^{2}+a = v_{k}(z_{1}).
\]
It follows that
\[
\frac{\de^{2}\vp}{\de z^{i}\de\ov{z}^{j}}(z_{1})+\ve\delta_{ij}
\geq \frac{\de^{2}v_{k}}{\de z^{i}\de\ov{z}^{j}}(z_{1}).
\]
On the other hand, by the definitions of $v_{k}$, $Z_{i}$ and \eqref{blowup argument eqn 1},
\begin{equation}\label{blowup argument eqn 2}
\begin{split}
\frac{\de^{2}v_{k}}{\de z^{i}\de\ov{z}^{j}}(z_{1})
= {} & \frac{1}{\Lambda_{i}^{2}}(Z_{i}\ov{Z}_{j}u_{k})
\left(p_{\infty}+\frac{z_{1}}{\Lambda_{k}}\right) \\
= {} & \frac{1}{\Lambda_{k}^{2}}\omega_{u_{k}}(e_{i},\ov{e}_{j})
\left(p_{\infty}+\frac{z_{1}}{\Lambda_{k}}\right)+o(1) \\
\in {} & \frac{1}{\Lambda_{k}^{2}}\Gamma+o(1) \\[1mm]
= {} & \Gamma+o(1).
\end{split}
\end{equation}
Increasing $k$ if necessary,
\[
\lambda\left[\frac{\de^{2}\vp}{\de z^{i}\de\ov{z}^{j}}(z_{1})\right]
\in \Gamma-2\ve\mathbf{1}.
\]
Letting $\ve\rightarrow0$, we obtain
\[
\lambda\left[\frac{\de^{2}\vp}{\de z^{i}\de\ov{z}^{j}}(z_{0})\right]
\in \ov{\Gamma}.
\]

\bigskip

(2) For any $\ve>0$, we choose sufficiently large $k$ such that
\[
\|v_{k}-v_{\infty}\|_{C^{0}(B_{1}(z_{0}))} \leq \ve^{2}.
\]
Then there exist $a\in\mathbb{R}$ with $|a|\leq\ve$ and $z_{1}\in B_{\ve}(z_{0})$ satisfying
\[
\vp-\ve|z-z_{0}|^{2}+a \leq v_{k} \ \ \text{on $B_{\ve}(z_{0})$}
\]
and
\[
\vp(z_{1})-\ve|z_{1}-z_{0}|^{2}+a = v_{k}(z_{1}).
\]
It follows that
\[
\frac{\de^{2}\vp}{\de z^{i}\de\ov{z}^{j}}(z_{1})-\ve\delta_{ij}
\leq \frac{\de^{2}v_{k}}{\de z^{i}\de\ov{z}^{j}}(z_{1}).
\]
Suppose that
\begin{equation}\label{blowup argument eqn 3}
\lambda\left[\frac{\de^{2}\vp}{\de z^{i}\de\ov{z}^{j}}(z_{1})\right]
\in \Gamma+3\ve\mathbf{1}.
\end{equation}
Then
\[
\lambda\left[\frac{\de^{2}v_{k}}{\de z^{i}\de\ov{z}^{j}}(z_{1})\right]
\in \Gamma+2\ve\mathbf{1}.
\]
Combining this with \eqref{blowup argument eqn 2},
\[
\frac{1}{\Lambda_{k}^{2}}\cdot\lambda\left[\omega_{u_{k}}(e_{i},\ov{e}_{j})
\left(p_{\infty}+\frac{z_{1}}{\Lambda_{k}}\right)\right]
= \lambda\left[\frac{\de^{2}v_{k}}{\de z^{i}\de\ov{z}^{j}}(z_{1})\right]+o(1)
\in \Gamma+\ve\mathbf{1}.
\]
This shows
\[
\lambda\left[\omega_{u_{k}}(e_{i},\ov{e}_{j})
\left(p_{\infty}+\frac{z_{1}}{\Lambda_{k}}\right)\right]
\in \Lambda_{k}^{2}\Gamma+\ve\Lambda_{k}^{2}\mathbf{1}
= \Gamma+\ve\Lambda_{k}^{2}\mathbf{1}.
\]
By $f(\lambda(\omega_{u_{k}}))=F(\omega_{u_{k}})=h_{k}$ and \cite[Lemma 9 (a)]{Szekelyhidi18}, we obtain
\[
h_{k}\left(p_{\infty}+\frac{z_{1}}{\Lambda_{k}}\right)
= f(\lambda\left[\omega_{u_{k}}(e_{i},\ov{e}_{j})
\left(p_{\infty}+\frac{z_{1}}{\Lambda_{k}}\right)\right]) \rightarrow \sup_{\Gamma}f \ \
\text{as $k\rightarrow\infty$}.
\]
If $\sup_{\Gamma}f=\infty$, then the above contradicts with $\|h_{k}\|_{C^{2}}\leq C_{0}$. If $\sup_{\Gamma}f<\infty$, then the above contradicts with
\[
h_{k}\leq\sup_{\Gamma}f-C_{0}^{-1}.
\]
This implies \eqref{blowup argument eqn 3} is impossible and so
\[
\lambda\left[\frac{\de^{2}\vp}{\de z^{i}\de\ov{z}^{j}}(z_{1})\right]
\notin \Gamma+3\ve\mathbf{1}.
\]
Letting $\ve\rightarrow0$, we obtain
\[
\lambda\left[\frac{\de^{2}\vp}{\de z^{i}\de\ov{z}^{j}}(z_{0})\right]
\in \mathbb{R}^{n}\setminus\Gamma.
\]
\end{proof}

Combining the above claim and $\|v_{\infty}\|_{C^{1,1}(\mathbb{C}^{n})}\leq C$ with \cite[Theorem 20]{Szekelyhidi18}, the function $v_{\infty}$ is constant, which contradicts with $|\de v_{\infty}|(0)=1$. Then we complete the proof of Theorem \ref{blowup argument}.
\end{proof}

\subsection{$C^{2,\alpha}$ estimate}

\begin{proof}[Proof of Theorem \ref{main estimate}]
For any $p\in M$, we choose a coordinate system $(U,\{x^{\alpha}\}_{\alpha=1}^{2n})$ such that at $p$,
\[
\chi_{\alpha\beta}
= \chi\left(\frac{\de}{\de x^{\alpha}},\frac{\de}{\de x^{\alpha}}\right)
= \delta_{\alpha\beta}
\]
and
\[
J\left(\frac{\de}{\de x^{2i-1}}\right) = \frac{\de}{\de x^{2i}}, \quad
\text{for $i=1,2,\ldots,n$}.
\]
We further assume that $U$ contains $B_{1}(0)\subset\mathbb{R}^{2n}$. In $B_{1}(0)$, we define the $(1,0)$-vector field $Z_{i}$ by
\[
Z_{i} = \frac{1}{\sqrt{2}}
\left(\frac{\de}{\de x^{i}}-\sqrt{-1}J\left(\frac{\de}{\de x^{i}}\right)\right), \quad \text{for $i=1,2,\cdots,n$}.
\]
It is clear that $\{Z_{i}(0)\}$ is a basis for $T_{p}^{(1,0)}M$. Shrinking $B_{1}(0)$ if necessary, we may assume that $\{Z_{i}\}$ is also a $(1,0)$-frame of $T_{\mathbb{C}}M$ on $B_{1}(0)$. Let $\theta^{i}$ be the dual frame of $\{Z_{i}\}_{i=1}^{n}$. Write
\[
\chi = \sqrt{-1}\chi_{i\ov{j}}\theta^{i}\wedge\ov{\theta}^{j}, \quad
\omega = \sqrt{-1}g_{i\ov{j}}\theta^{i}\wedge\ov{\theta}^{j}, \quad
\omega_{u} = \sqrt{-1}\ti{g}_{i\ov{j}}\theta^{i}\wedge\ov{\theta}^{j}.
\]

In order to apply \cite[Theorem 1.2]{TWWY15}, we first define $T$, $S$ and $\mathcal{G}$ (since the letter $F$ has already been used in \eqref{nonlinear equation}, here we use $\mathcal{G}$ instead). Let $\mathrm{Sym}(2n)$ be the space of symmetric $2n\times 2n$ matrices with real entries. Define the map $T:\mathrm{Sym}(2n)\times B_{1}(0)\rightarrow\mathrm{Sym}(2n)$ by
\[
T(N,x) = \frac{1}{4}\left(N+J^{T}(x)\cdot N\cdot J(x)\right).
\]
Let $H(u)$ be the symmetric bilinear form given by
\[
H(u)(X,Y)=(dJdu)^{(1,1)}(X,JY)=2\ddbar u(X,JY),
\]
By \cite[p. 443]{TWWY15}, we have
\[
H(u)(x) = 2T(D^{2}u(x),x)+E(u)(x),
\]
where $E(u)(x)$ is an error matrix which depends linearly on $Du(x)$. Define the map $S: B_{1}(0)\rightarrow\mathrm{Sym}(2n)$ by
\[
 S(x) = g(x)+\frac{1}{2}E(u)(x).
\]
Then
\[
S(x)+T(D^{2}u(x),x) = g(x)+\frac{1}{2}H(u)(x)
\]
and

\begin{equation}\label{ti g i ov j}
\begin{split}
& \big[S(x)+T(D^{2}u(x),x)\big](Z_{i},\ov{Z}_{j})
= g_{i\ov{j}}+\frac{1}{2}H(u)(Z_{i},\ov{Z}_{j}) \\
= {} & g_{i\ov{j}}+(\sqrt{-1}\de\dbar u)(Z_{i},J\ov{Z}_{j})
= g_{i\ov{j}}+(\de\dbar u)(Z_{i},\ov{Z}_{j}) = \ti{g}_{i\ov{j}}.
\end{split}
\end{equation}

For any $N\in\mathrm{Sym}(2n)$, we write
\[
N_{i\ov{j}}(x) = N(Z_{i}(x),\ov{Z}_{j}(x)).
\]
Then $(N_{i\ov{j}}(x))$ is a Hermitian matrix. Let $\mu(N,x)$ be the eigenvalues of $(N_{i\ov{j}}(x))$ with respect to $\chi(x)$. Consider the set
\[
\mathcal{E} = \{N\in\mathrm{Sym}(2n)~|~\mu(N,0)\in\ov{\Gamma^{\sigma}}\cap\ov{B_{R}(0)}\},
\]
where $\sigma$ and $R>0$ are two constants determined later. By the assumptions of $f$ and $\Gamma$, we see that $\mathcal{E}$ is compact and convex.

After shrinking $B_{1}(0)$, we may assume that $\chi(x)$ is close to $\chi(0)$, which implies $\mu(N,x)$ is close to $\mu(N,0)$. Thanks to Proposition \ref{blowup argument}, choosing $\sigma$ and $R$ appropriately, we have
\[
S(x)+T(D^{2}u(x),x) \in \mathcal{E}, \ \ \text{on $B_{1}(0)$}.
\]
Furthermore, there exists a neighborhood $\mathcal{U}$ of $\mathcal{E}$ such that $\mu(N,x)\in\Gamma$ for any $(N,x)\in \mathcal{U}\times B_{1}(0)$. Define the function $\mathcal{G}:\mathcal{U}\times B_{1}(0)\rightarrow\mathbb{R}$ by
\[
\mathcal{G}(N,x) := f(\mu(N,x)).
\]
Extend it smoothly to $\mathrm{Sym}(2n)\times B_{1}(0)$, and still denote it by $\mathcal{G}$ for convenience. Combining \eqref{ti g i ov j} and the definition of $\mathcal{G}$, it is clear that
\[
\mathcal{G}(S(x)+T(D^{2}u(x),x),x) = F(\omega_{u}).
\]
Then on $B_{1}(0)$, the equation \eqref{nonlinear equation} can be written as
\[
\mathcal{G}(S(x)+T(N,x),x) = h(x).
\]

The assumptions of $f$ and $(M,\chi,J)$ show $\mathcal{G}$, $T$ satisfies \cite[{\bf H1}, \bf {H2}]{TWWY15}. In the definition of $S$, the term $E(u)$ depends on $Du$ linearly. Proposition \ref{blowup argument} gives a uniform $C^{1}$ bound of this term. Then $S$ satisfies \cite[\bf {H3}]{TWWY15}. By \cite[Theorem 1.2]{TWWY15} and a covering argument, we obtain the required $C^{2,\alpha}$ estimate.
\end{proof}

\section{Proofs of Theorem \ref{complex Hessian equation} and \ref{n-1 MA equation}}\label{applications}
\subsection{Complex Hessian equation}
For $1\leq k\leq n$, let $\sigma_{k}$ and $\Gamma_{k}$ be the $k$-th elementary symmetric polynomial and the $k$-th G{\aa}rding cone on $\mathbb R^n$, respectively. Namely, for any $\mu=(\mu_{1},\mu_{2},\cdots,\mu_{n})\in\mathbb{R}^{n}$, we have
\begin{equation*}
\sigma_{k}(\mu) =\sum_{1\leq i_{1}<\cdots<i_{k}\leq n}\mu_{i_{1}}\mu_{i_{2}}\cdots\mu_{i_{k}}
\end{equation*}
and
\begin{equation*}
\Gamma_{k} = \{ \mu\in\mathbb{R}^{n}: \text{$\sigma_{i}(\mu)>0$ for $i=1,2,\cdots,k$} \}.
\end{equation*}
The above definitions can be extended to almost Hermitian manifold $(M,\chi,J)$ as follows:
\begin{equation*}
\sigma_{k}(\alpha)=\binom{n}{k}
\frac{\alpha^{k}\wedge\chi^{n-k}}{\chi^{n}}, \quad \text{for any $\alpha\in A^{1,1}(M)$},
\end{equation*}
and
\begin{equation*}
\Gamma_{k}(M,\chi) = \{ \alpha\in A^{1,1}(M): \text{$\sigma_{i}(\alpha)>0$ for $i=1,2,\cdots,k$} \}.
\end{equation*}
We say that $\alpha\in A^{1,1}(M)$ is $k$-positive if $\alpha\in\Gamma_{k}(M,\chi)$.

Now we are in position to prove Theorem \ref{complex Hessian equation}.
\begin{proof}[Proof of Theorem \ref{complex Hessian equation}]
The uniqueness of solution follows from the maximum principle (see e.g. \cite[p. 1980-1981]{CTW19}).
For the existence of solution, we consider the family of equations
\[
\begin{cases}
\ \omega_{u_{t}}^{k}\wedge \chi^{n-k}=e^{th+(1-t)h_{0}+c_{t}}\chi^{n}, \\[1mm]
\ \omega_{u_{t}}\in \Gamma_{k}(\chi), \\[0.5mm]
\ \sup_{M}u_{t}=0,
\end{cases}
\eqno (*)_{t}
\]
where $\{c_{t}\}$ are constants and
\[
h_{0} = \log\frac{\omega^{k}\wedge \chi^{n-k}}{\chi^{n}}.
\]
Let us define
\[
S = \{t\in[0,1]: \text{there exists a pair $(u_{t},c_{t})\in C^{\infty}(M)\times\mathbb{R}$ solving $(*)_{t}$}\}.
\]
Note that $(0,0)$ solves $(*)_{0}$ and hence $S\neq\emptyset$. To prove the existence of solution for \eqref{complex Hessian equation 1}, it suffices to prove that $S$ is closed and open.

\bigskip
\noindent
{\bf Step 1.} $S$ is closed.
\bigskip

We first show that $\{c_{t}\}$ is uniformly bounded. Suppose that $u_{t}$ achieves its maximum at the point $p_{t}\in M$. Then the maximum principle shows $\ddbar u_{t}$ is non-positive at $p_{t}$. Combining this with $(*)_{t}$, we obtain the upper bound of $c_{t}$:
\[
c_{t} \leq -th(p_{t})+th_{0}(p_{t}) \leq C,
\]
for some $C$ depending only on $h$, $\omega$ and $\chi$. The lower bound of $c_{t}$ can be proved similarly.

Define
\[
f = \log\sigma_{k}, \quad \Gamma = \Gamma_{k}.
\]
As pointed out in \cite[p. 368-369]{Szekelyhidi18}, one can verify that the above setting satisfies the structural assumptions in Section \ref{introduction}. Furthermore, the $k$-positivity of $\omega$ shows $\underline{u}=0$ is a $\mathcal{C}$-subsolution of $(*)_{t}$. Then $C^{\infty}$ a priori estimates of $u_{t}$ follows from Theorem \ref{main estimate} and the standard bootstrapping argument. Combining this with the Arzel\`a-Ascoli theorem, $S$ is closed.

\bigskip
\noindent
{\bf Step 2.} $S$ is open.
\bigskip

Suppose there exists a pair $(u_{\hat{t}},c_{\hat{t}})$ satisfies  $(*)_{\hat{t}}$. We hope to show that when $t$ is close to $\hat{t}$, there exists a pair $(u_{t},c_{t})\in C^{\infty}(M)\times\mathbb{R}$ solving $(*)_{t}$.  	

We first consider the linearized operator of $(*)_{\hat{t}}$ at $u_{\hat{t}}$:
\[
L_{u_{\hat{t}}}(\psi):=k\frac{\ddbar\psi\wedge\omega_{u_{\hat{t}}}^{k-1}\wedge\chi^{n-k}}{\omega_{u_{\hat{t}}}^{k}\wedge \chi^{n-k}}, \quad \text{for $\psi\in C^{2}(M)$}.
\]
The operator $L_{u_{\hat{t}}}$ is elliptic and then the index is zero. By the maximum principle,
\begin{equation}\label{ker L1}
\mathrm{Ker}(L_{u_{\hat{t}}})=\{\text{constant}\}.
\end{equation}
Denote $L_{u_{\hat{t}}}^{*}$ by the $L^{2}$-adjoint operator of $L_{u_{\hat{t}}}$ with respect to the volume form
\[
\mathrm{dvol}_{k}=\omega_{u_{\hat{t}}}^{k}\wedge \chi^{n-k}.
\]
By the Fredholm alternative, there is a non-negative function $\xi$ such that
\begin{equation}\label{ker L*}
\mathrm{Ker}(L_{u_{\hat{t}}}^{*})=\text{Span}\big\{\xi\big\}.
\end{equation}
It follows from the strong maximum principle that $\xi>0$. Up to a constant, we may and do assume
\[
\int_{M}\xi\, \mathrm{dvol}_{k}=1.
\]

To prove the openness of $S$, we define the space  by
\[
\mathcal{U}^{2,\alpha}:=\Big\{\phi\in C^{2,\alpha}(M):
\text{$\omega_{\phi}\in\Gamma_{k}(\chi)$ and
$\int_{M}\phi\cdot\xi \, \mathrm{dvol}_{k}=0$}
\Big\}.
\]
Then the tangent space of $\mathcal{U}^{2,\alpha}$ at $u_{\hat{t}}$ is given by
\[
T_{u_{\hat{t}}}\,\mathcal{U}^{2,\alpha}
:=\Big\{\psi\in C^{2,\alpha}(M): \int_{M}\psi\cdot\xi \, \mathrm{dvol}_{k}=0\Big\}.
\]
Let us consider the map
\begin{equation*}
G(\phi,c)=\log\sigma_{k}(\omega_{\phi})-c,
\end{equation*}
which maps $\mathcal{U}^{2,\alpha}\times\mathbb{R}$ to $C^{\alpha}(M)$.
It is clear that the linearized operator of $G$ at $(u_{\hat{t}},\hat{t})$ is given by
\begin{equation}\label{linear operator}
(L_{u_{\hat{t}}}-c): T_{u_{\hat{t}}}\,\mathcal{U}^{2,\alpha}\times \mathbb{R}\longrightarrow  C^{\alpha}(M).
\end{equation}
On the one hand, for any $h\in C^{\alpha}(M)$, there exists a constant $c$ such that
\[
\int_{M}(h+c)\cdot\xi \, \text{dvol}_{k}=0.
\]
By \eqref{ker L*} and the Fredholm alternative, we can find a real function $\psi$ on $M$ such that
\[
L_{u_{\hat{t}}}(\psi)=h+c
\]
Hence, the map $L_{u_{\hat{t}}}-c$ is surjective. On the other hand, suppose that there are two pairs $(\psi_{1},c_{1}),(\psi_{2},c_{2})\in T_{u_{\hat{t}}}\,\mathcal{U}^{2,\alpha}\times \mathbb{R}$ such that
\[
L_{u_{\hat{t}}}(\psi_{1})-c_{1}
= L_{u_{\hat{t}}}(\psi_{2})-c_{2}.
\]
It then follows that
\[
L_{u_{\hat{t}}}(\psi_{1}-\psi_{2}) = c_{1}-c_{2}.
\]
Applying the maximum principle twice, we obtain $c_{1}=c_{2}$ and $\psi_{1}=\psi_{2}$. Then $L_{u_{\hat{t}}}-c$ is injective.

Now we conclude that $L_{u_{\hat{t}}}-c$ is bijective. By the inverse function theorem, when $t$ is close to $\hat{t}$, there exists a pair $(u_{t},c_{t})\in\mathcal{U}^{2,\alpha}\times\mathbb{R}$ satisfying
\begin{equation*}
 G(u_{t},c_{t}) = th+(1-t)h_{0}.
\end{equation*}
The standard elliptic theory shows that $u_{t}\in C^{\infty}(M)$. Then $S$ is open.
\end{proof}

\subsection{The Monge-Ampere equation for $(n-1)$ plurisubharmonic functions}
As in \cite[p. 371-372]{Szekelyhidi18}, let $T$ be the linear map given by
\[
T(\mu) = (T(\mu)_{1},\ldots,T(\mu)_{n}), \quad
T(\mu)_{k} = \frac{1}{n-1}\sum_{i\neq k}\mu_{i}, \quad
\text{for $\mu\in\mathbb{R}^{n}$}.
\]
Define
\[
f = \log\sigma_{n}(T), \quad \Gamma = T^{-1}(\Gamma_{n}).
\]
It is straightforward to verify that the above setting satisfies the assumptions (i)-(iii) in Section \ref{introduction}. Write
\[
\omega =(\mathrm{tr}_{\chi}\eta)\chi-(n-1)\eta,
\]
and so the equation in \eqref{n-1 MA equation 1} can be written as
\[
F(\omega_{u}) = h+c.
\]
Then Theorem \ref{n-1 MA equation} can be proved by the similar argument of Theorem \ref{complex Hessian equation}.


\begin{thebibliography}{99}

\bibitem{Andrews94} Andrews, B. {\em Contraction of convex hypersurfaces in Euclidean space}, Calc. Var. Partial Differential Equations {\bf 2} (1994), no. 2, 151--171.

\bibitem{Blocki05} B\l ocki, Z. {\em On uniform estimate in Calabi-Yau theorem}, Sci. China Ser. A {\bf 48} (2005), suppl., 244--247.

\bibitem{Blocki11} B\l ocki, Z. {\em On the uniform estimate in the Calabi-Yau theorem, II}, Sci. China Math. {\bf 54} (2011), no. 7, 1375--1377.

\bibitem{CNS85} Caffarelli, L., Nirenberg, L., Spruck, J. {\em The Dirichlet problem for nonlinear second-order elliptic equations. III. Functions of the eigenvalues of the Hessian}, Acta Math. {\bf 155} (1985), no. 3-4, 261--301.

\bibitem{Calabi57} Calabi, E. {\em On K\"ahler manifolds with vanishing canonical class}, Algebraic geometry and topology. A symposium in honor of S. Lefschetz, pp. 78--89. Princeton University Press, Princeton, N. J., 1957.

\bibitem{Cherrier87} Cherrier, P. {\em \'Equations de Monge--Amp\`ere sur les vari\'et\'es Hermitiennes compactes}, Bull. Sc. Math. (2) {\bf 111} (1987), 343--385.

\bibitem{CM21} Chu, J., McCleerey, N. {\em Fully non-linear degenerate elliptic equations in complex geometry}, J. Funct. Anal. {\bf 281} (2021), no. 9,  Paper No. 109176, 45 pp.

\bibitem{CTW19} Chu, J., Tosatti, V., Weinkove, B. {\em The Monge--Amp\`ere equation for non-integrable almost complex structures}, J. Eur. Math. Soc. (JEMS) {\bf 21} (2019), no. 7, 1949--1984.

\bibitem{DK17} Dinew, S., Ko\l odziej, S. {\em Liouville and Calabi-Yau type theorems for complex Hessian equations}, Amer. J. Math. {\bf 139} (2017), no. 2, 403--415.

\bibitem{EH89} Ecker, K., Huisken, G. {\em Immersed hypersurfaces with constant Weingarten curvature}, Math. Ann. {\bf 283} (1989), no. 2, 329--332.

\bibitem{FLM11} Fang, H., Lai, M., Ma, X.-N. {\em On a class of fully nonlinear flows in K\"ahler geometry}, J. Reine Angew. Math. {\bf 653} (2011), 189--220.

\bibitem{FWW10} Fu, J., Wang, Z., Wu, D. {\em Form-type Calabi-Yau equations}, Math. Res. Lett. {\bf 17} (2010), no. 5, 887--903.

\bibitem{FWW15} Fu, J., Wang, Z., Wu, D. {\em Form-type Calabi-Yau equations on K\"ahler manifolds of nonnegative orthogonal bisectional curvature}, Calc. Var. Partial Differential Equations {\bf 52} (2015), no. 1-2, 327--344.

\bibitem{Gerhardt96} Gerhardt, C. {\em Closed Weingarten hypersurfaces in Riemannian manifolds}, J. Differential Geom. {\bf 43} (1996), no. 3, 612--641.

\bibitem{Guan14} Guan, B. {\em Second-order estimates and regularity for fully nonlinear elliptic equations on Riemannian manifolds}, Duke Math. J. {\bf 163} (2014), no. 8, 1491--1524.

\bibitem{GL10} Guan, B., Li, Q. {\em Complex Monge--Amp\`ere equations and totally real submanifolds}, Adv. Math. {\bf 225} (2010), no. 3, 1185--1223.

\bibitem{GS15} Guan, B., Sun, W. {\em On a class of fully nonlinear elliptic equations on Hermitian manifolds}, Calc. Var. Partial Differential Equations {\bf 54} (2015), no. 1, 901--916.

\bibitem{Hanani96} Hanani, A. {\em \'{E}quations du type de Monge--Amp\`{e}re sur les vari\'{e}t\'{e}s hermitiennes compactes}, J. Funct. Anal. {\bf 137} (1996), no.1, 49--75.

\bibitem{HL15} Harvey, F.R., Lawson, H.B. {\em Potential theory on almost complex manifolds}, Ann. Inst. Fourier (Grenoble) {\bf 65} (2015), no. 1, 171--210.

\bibitem{Hou09} Hou, Z. {\em Complex Hessian equation on K\"ahler manifold}, Int. Math. Res. Not. IMRN 2009, no. 16, 3098--3111.

\bibitem{HMW10} Hou, Z., Ma, X.-N., Wu, D. {\em A second order estimate for complex Hessian equations on a compact K\"ahler manifold}, Math. Res. Lett. {\bf 17} (2010), no. 3, 547--561.

\bibitem{Li14} Li, Y. {\em A priori estimates for Donaldson's equation over compact Hermitian manifolds}, Calc. Var. Partial Differential Equations {\bf 50} (2014), no. 3-4, 867--882.

\bibitem{SW08} Song, J., Weinkove, B. {\em On the convergence and singularities of the $J$-flow with applications to the Mabuchi energy}, Comm. Pure Appl. Math. {\bf 61} (2008), no. 2, 210--229.

\bibitem{Spruck05} Spruck, J. {\em Geometric aspects of the theory of fully nonlinear elliptic equations}, Global theory of minimal surfaces, Amer. Math. Soc., Providence, RI, 2005, 283--309.

\bibitem{Sun16} Sun, W. {\em On a class of fully nonlinear elliptic equations on closed Hermitian manifolds}, J. Geom. Anal. {\bf 26} (2016), no. 3, 2459--2473.

\bibitem{Sun17a} Sun, W. {\em On a class of fully nonlinear elliptic equations on closed Hermitian manifolds II: $L^{\infty}$ estimate}, Comm. Pure Appl. Math. {\bf 70} (2017), no. 1, 172--199.

\bibitem{Sun17b} Sun, W. {\em On uniform estimate of complex elliptic equations on closed Hermitian manifolds}, Commun. Pure Appl. Anal. {\bf 16} (2017), no. 5, 1553--1570.

\bibitem{Szekelyhidi18} Sz\'ekelyhidi, G. {\em Fully non-linear elliptic equations on compact Hermitian manifolds}, J. Differential Geom. {\bf 109} (2018), no. 2, 337--378.

\bibitem{TWWY15} Tosatti, V., Wang, Y., Weinkove, B., Yang, X. {\em $C^{2,\alpha}$ estimates for nonlinear elliptic equations in complex and almost complex geometry}, Calc. Var. Partial Differential Equations {\bf 54} (2015), no. 1, 431--453.

\bibitem{TW10a} Tosatti, V., Weinkove, B. {\em Estimates for the complex Monge--Amp\`ere equation on Hermitian and balanced manifolds}, Asian J. Math. {\bf 14} (2010), no. 1, 19--40.

\bibitem{TW10b} Tosatti, V., Weinkove, B. {\em The complex Monge--Amp\`ere equation on compact Hermitian manifolds}, J. Amer. Math. Soc. {\bf 23} (2010), no. 4, 1187--1195.

\bibitem{TW17} Tosatti, V. and Weinkove, B., {\em The Monge--Amp\`ere equation for $(n-1)$-plurisubharmonic functions on a compact K\"ahler manifold}, J. Amer. Math. Soc. {\bf 30} (2017), no. 2, 311--346.

\bibitem{TW19} Tosatti, V. and Weinkove, B. {\em Hermitian metrics, $(n-1, n-1)$ forms and Monge--Amp\`ere equations}, J. Reine Angew. Math. {\bf 755} (2019), 67--101.

\bibitem{Yau78} Yau, S.-T. {\em On the Ricci curvature of a compact K\"ahler manifold and the complex Monge--Amp\`ere equation, I}, Comm. Pure Appl. Math. {\bf 31} (1978), no. 3, 339--411.

\bibitem{Zhang17} Zhang, D. {\em Hessian equations on closed Hermitian manifolds}, Pacific J. Math. {\bf 291} (2017), no. 2, 485--510.

\bibitem{ZJ21} Zhang, J. {\em Monge--Amp\`ere type equations on almost Hermitian manifolds}, preprint, arXiv:2101.00380.

\bibitem{ZZ11} Zhang, X., Zhang, X. {\em Regularity estimates of solutions to complex Monge--Amp\`ere equations on Hermitian manifolds}, J. Funct. Anal. {\bf 260} (2011), no. 7, 2004--2026.


\end{thebibliography}
\end{document}